\documentclass[12pt,a4paper,oneside]{amsart}
\usepackage{amsfonts, amsmath, amssymb, amsthm, amscd}
\usepackage{anysize}
\usepackage{mathtools}
\marginsize{2cm}{2cm}{1.5cm}{1.5cm}

\usepackage[T2A]{fontenc}
\usepackage[utf8]{inputenc}
\usepackage[english]{babel}
\usepackage{url}
\usepackage{cmap}
\usepackage{enumerate}

\usepackage{xcolor}
\usepackage{hyperref}
\hypersetup{
	colorlinks   = true, %Colours links instead of ugly boxes
	urlcolor     = blue, %Colour for external hyperlinks
	linkcolor    = blue, %Colour of internal links
	citecolor    = red %Colour of citations
}
\newcommand{\Href}[2]{\hyperref[#2]{#1~\ref{#2}}}

\usepackage[ruled,vlined]{algorithm2e}

\newtheorem{thm}{Theorem}[section]
\newtheorem{cor}[thm]{Corollary}
\newtheorem{lem}[thm]{Lemma}
\newtheorem{prp}[thm]{Proposition}
\newtheorem{clm}[thm]{Claim}
\theoremstyle{definition}
\newtheorem{dfn}[thm]{Definition}
\theoremstyle{remark}

\newtheorem{rem}[thm]{Remark}
\newtheorem*{sketch}{Sketch of the proof}

\usepackage{xcolor}

\newcommand{\st}{:\;}

\newcommand{\norm}[1]{\left\|#1\right\|}
\newcommand{\enorm}[1]{\left|#1\right|}

\newcommand{\iprod}[2]{\left\langle#1,#2\right\rangle}%

\newcommand{\N}{\mathbb N}
\newcommand{\R}{\mathbb R}

\renewcommand{\phi}{\varphi}
\renewcommand{\epsilon}{\varepsilon}

\newcommand{\dist}[2]{\operatorname{dist}\!\left( #1, #2 \right)}%

\newcommand{\e}{\varepsilon}

\newcommand{\littleo}[1]{o \! \parenth{#1}}%

\providecommand{\parenth}[1]{\left(#1\right)}%

\newcommand{\partvar}[2][1]{\mathrm{S}_{#1} \! \parenth{#2}}
\newcommand{\variation}[2][1]{\mathrm{V}_{#1} \! \parenth{#2}}
\newcommand{\variationsq}[1]{\mathrm{V}_{2}^2 \! \parenth{#1}}
\newcommand{\partvarsq}[1]{\mathrm{S}_{2}^2 \! \parenth{#1}}
\providecommand{\abs}[1]{\lvert#1\rvert}%

\newcommand{\length}[1]{\mathrm{length}\,  #1 }

\newcommand{\anglef}[2]{\measuredangle \! \parenth{#1, #2}}%
\newcommand{\strconvsegment}[2]{D_R \! \left[#1, #2 \right]}%
\DeclareMathOperator{\mesh}{mesh}
\DeclareMathOperator{\hilbert}{\mathcal{H}}
\DeclareMathOperator{\sphere}{\mathbb{S}}

\newcommand{\di}{\,\mathrm{d}}%integration d

\title{Shortest curves in  proximally smooth sets: existence and uniqueness}
\author{Grigory M. Ivanov\address{Grigory M. Ivanov: 
Institute of Science and Technology Austria (IST Austria), 
Klosterneuburg, 3400, Austria}
\email{grimivanov@gmail.com}%
\and
 Mariana S. Lopushanski\address{ Mariana S. Lopushanski:
 Steklov Mathematical Institute of the Russian Academy of Sciences,  Moscow,
119991, Russia}
\email{lopushanski@phystech.edu}%
\and
Grigorii E. Ivanov\address{Grigorii E. Ivanov:
Moscow Institute of Physics and Technology, Dolgoprudny,  141701, Russia}
\email{g.e.ivanov@gmail.com}
}
\thanks{Section 4 was obtained at Steklov Mathematical Institute RAS and supported by the Russian Science Foundation, grant N 19-11-00087,   \url{https://rscf.ru/en/project/19-11-00087/}}

\begin{document}
\begin{abstract}
 We study  shortest curves in proximally smooth subsets of a Hilbert space.  We consider an $R$-proximally smooth set $A$ in a Hilbert space with points  $a$ and $b$ satisfying $\enorm{a-b} < 2R.$ 
We provide a simple geometric algorithm of constructing a curve inside  $A$ connecting $a$ and $b$ whose length is at most $2R \arcsin\frac{\enorm{a-b}}{2R},$ which corresponds to the shortest curve inside the model set -- a Euclidean sphere of radius $R$ passing through $a$ and $b.$  Using this construction, we  show that there exists a unique shortest curve inside $A$ connecting $a$ and $b.$  This result is tight since two points of $A$ at distance $2R$ are not necessarily connected in $A;$   the bound on the length cannot be improved since the equality is attained on the Euclidean sphere of radius $R.$
\end{abstract}
\maketitle
\section{Introduction}

In non-smooth analysis, one usually considers a function with a certain generalized sub-differentiability property. In turn, such classes of functions can be equivalently described in geometric terms via the geometric properties of  their epigraphs. A notable example is the class of \emph{lower-$C^2$} functions studied by R.~T. Rockafellar \cite{rockafellar1981favorable} and the class of proximally smooth sets \cite{Clarke1995}. 
We recall that a closed set $A$ of a normed space is \emph{$R$-proximally smooth} if the distance function $\dist{\cdot}{A}$  is
continuously differentiable on the \emph{open  $R$-neighborhood} of the form
$\{x \st 0  < \dist{x}{A} < R\}.$
Proximally smooth sets in Hilbert and Banach spaces have been studied extensively; some  historical notes  can be found in \cite{Pol_Rock_Thi1} and in \cite{Colombo_thibault}. An extensive study of proximally smooth sets and their properties in a Hilbert space can be found in \cite{IvMonEng}. 
We note that there are a dozen notions equivalent to proximal smoothness in a Hilbert space, often referred to as \emph{weakly convex sets}. Unfortunately, most of them are not equivalent in general Banach spaces (even uniformly smooth and uniformly convex), as shown in \cite{IvBal2009,Ivanov:233858}.

In this paper, we consider only the Euclidean case. Our main result is as follows.  

\begin{thm}\label{thm:existence of the shortest_curve}
 Let $a$ and $b$  be two distinct points satisfying $\enorm{a-b} < 2R$ of $A$, where $A$ is an $R$-proximally smooth set.  Then the shortest curve in $A$ between $a$ and $b$ exists and is unique. Moreover, the shortest curve is an $R$-proximally smooth set, whose length is at most the length of the shortest arc of a circle of radius $R$ passing through points $a$ and $b,$ i.e.
$2R \arcsin \frac{\enorm{a-b}}{2R}.$
\end{thm}

We will present a geometric approach to this result. We note that earlier it was obtained by G.~E. Ivanov \cite[Theorem 1.14.1]{IvMonEng} in Russian, but has never been published in a peer-reviewed journal. However, our proof differs from the initial one in several important details. First, we provide a simple and explicit algorithm for constructing a curve inside a proximally smooth set connecting two of its points $a$ and $b$, whose length is at most $2R \arcsin \frac{\enorm{a-b}}{2R}.$ In particular, we will provide a purely geometric proof of the following lemma.

\begin{lem}
\label{lem:short_curve_intro} 
Consider two distinct points  $a$ and $b$  satisfying $\enorm{a-b} < 2R$ of  an $R$-proximally smooth set $A$. Then there exists a curve $\gamma$ in $A$ connecting $a$ and $b$, whose length is at most
 $2R \arcsin \frac{\enorm{a-b}}{2R}.$
\end{lem}
Our key observation here is that to get the upper bound on the length of the polyline, one shall follow the length of arcs contended by the segments of a polyline rather than the length of the polyline itself.

Note that the assumptions in \Href{Theorem}{thm:existence of the shortest_curve} cannot be improved. Let us discuss the following simple examples. Consider the union of two parallel hyperplanes on distance $2R$. It is an $R$-proximally smooth set. However, the points from the hyperplanes are not connected. Another trivial example is a sphere of radius $R$. It is also proximally smooth, but many shortest curves connect two diametrical points of the sphere.

Second, we will show a slightly stronger statement than \Href{Theorem}{thm:existence of the shortest_curve}. Namely, we will show
that any length-minimizing sequence of curves in $A$ converges to the shortest path. Again, our geometric construction of curves allows us to do so without cumbersome technical details.

Third, we will explain  how to switch from the length functional to a certain discrete energy functional and backward. 
Roughly speaking, we propose an equivalent of finding the shortest path discrete problem that one may hope to resolve computationally.

For the sake of completeness, we will provide  a geometric proof of the
following somewhat known to specialists in the differential geometry result (also, see \cite[Theorem 1.14.2]{IvMonEng}).
We derive the following theorem from \Href{Lemma}{lem:short_curve_intro}.
\begin{thm}\label{thm:characterization_of_prox_reg_curves}
Let $\Gamma\colon [0,\length{\Gamma}] \to \hilbert$ be a rectifiable curve such that $\length\Gamma\leq\pi R$ and without self-intersections given by its natural parametrization connecting points $a$ and $b$ such 
that $0 < \enorm{a-b} < 2R.$ Then the following assertion are equivalent:
\begin{enumerate}
    \item\label{i1:thm_curves_char} $\Gamma$ is an $R$-proximally smooth   set;
    \item\label{i2:thm_curves_char} the natural parametrization of $\Gamma$ is differentiable and its derivative is Lipschitz  with constant $\frac{1}{R}.$
\end{enumerate}
\end{thm}
Here, as usual, we understand by the \emph{natural parametrization} of a rectifiable curve its arc-length parametrization.

The rest of the paper is organized as follows. 
We settle all our notations at the beginning of the next section except for the definition of ``folding mapping'' in \Href{Section}{sec:strongly_convex_segment}. 
In \Href{Subsection}{subsec:curves}, we give all the necessary definitions related to curves and their length and prove several simple standard lemmas related to variations. Next, in \Href{Subsection}{subsec:proximally_smooth_set}, we mention several properties equivalent to proximal smoothness. They are purely geometric, and we will work with those equivalent descriptions of proximal smoothness. One such description is via strongly convex segments (or ``spindles''). All these properties of proximally smooth sets are not new; nevertheless, we will provide the proofs for some of the facts for completeness.  We will proceed with a few metric properties of strongly convex segments in 
\Href{Section}{sec:strongly_convex_segment}; these are simple technical tools we will use later in the paper. After that, we will be ready to prove \Href{Lemma}{lem:short_curve_intro} together with some corollaries and its stability version in \Href{Section}{sec:short_curves}. We note that up to this part, all our proofs are simple. 
In \Href{Section}{sec:integration_along_circles}, we will obtain a certain discrete version of the classical Wirtinger inequality, which is used in the proof of \Href{Theorem}{thm:existence of the shortest_curve}.
\Href{Section}{sec:discrete_energy_functional} is dedicated to the proof of \Href{Theorem}{thm:existence of the shortest_curve}.
 We will introduce a certain ``discrete energy functional''  and argue that its minimizers are sufficiently fine approximations of shortest curves in a proximally smooth set. In the core, the main idea is an ``averaging'' of two curves. We will show that a certain geometric construction (actually, it will be the metric projection of the pointwise average of two curves) returns a curve shorter than the two initial ones. 
Finally, we will prove \Href{Theorem}{thm:characterization_of_prox_reg_curves} in the last 
\Href{Section}{sec:prox_smooth_curve_characterization}. We note that the proof of 
\Href{Theorem}{thm:characterization_of_prox_reg_curves} requires only results from the first four sections of the paper.

\section{Definitions and preliminaries}
\label{sec:definitions}
We use $[n]$ to denote the set $\{1, \dots, n\}$ for some natural $n.$ The closed line segment between two points $a$ and $b$ of a linear space is denoted by $[a,b].$ We use $B_R(x)$ to denote a ball of radius $R$ and center $x$ in a normed space.
Throughout the paper, we use $\hilbert$  to denote an at least two-dimensional Hilbert space, $\enorm{x}$ to denote the Euclidean norm of $x \in \hilbert,$ $\iprod{x}{y}$ to denote the inner product of $x,y\in\hilbert$, 
$\sphere_{\hilbert}$ to denote the unit sphere of $\hilbert.$ The boundary of a set $A$ is denoted by $\partial A.$     The {\it distance} from a point $x$ of a normed space $\parenth{X, \norm{\cdot}}$ to a set $A $ is defined  by 
$\dist{x}{A} = \inf\limits_{a \in A} \norm{x - a}.
$ 
The \emph{metric projection} of a point $x$ onto a set $A$ is the set of all points $a$ of $A$ that the identity $\norm{a - x} = \dist{x}{A}$ holds.
We call the set 
$
\left\{ x \in X \st 0 < \dist{x}{A} < R \right\}
$ 
the \emph{open $R$-neighborhood} of a subset $A$ of a normed space $\parenth{X, \norm{\cdot}}.$
Here and throughout the paper, we will assume that a strictly positive constant 
$R$ is fixed.
We also will extensively use the following geometric object.
\begin{dfn}
 For any two points $a$ and  $b$ of a normed space 
 $\parenth{X, \norm{\cdot}}$ with 
 $\norm{a - b} \leq 2R,$ we use  
$\strconvsegment{a}{b}$ to denote  the \emph{ strongly convex segment $[a,b]$ of radius}
$R,$ i.e. the set
\[
\strconvsegment{a}{b}=\bigcap\limits_{x\in X \colon \, a,b \in B_R(x)} B_R(x).
\]
% We also use $D_R(a,b)$ to denote the set 
% $\strconvsegment{a}{b}\setminus \{a,b\}.$
\end{dfn}
\begin{rem}\label{rem:str_convex_segment}
    We notice that for any $c,d\in D_R[a,b]$ the strongly convex segment $D_R[c,d]$ is contained in $D_R[a,b]$. We call this {\it the hereditary property of the strongly convex segment}.
\end{rem}
We shall note that the ``strongly convex segment'' disguises under the name ``spindle'' \cite{bezdek2007ball} in discrete geometry. It is used to define strongly convex sets with many different names: spindle convex, hyper-convex, etc.
As we will see in this section, there is a natural definition of proximally smooth subsets of a Euclidean space via strongly convex segments as well.

\begin{dfn}
We define the \emph{angle}, denoted $\anglef{a}{b},$ between two vectors $a$ and $b$ of $\hilbert$ as follows.
If at least one of the vectors is zero, then $\anglef{a}{b} = 0.$ Otherwise, 
$\anglef{a}{b}$ equals the length of the shortest arc of the great circle passing through  unit vectors $\frac{a}{\enorm{a}}$ and 
$\frac{b}{\enorm{b}}.$ 
\end{dfn}

Since we will use the set $R \sphere_{\hilbert}$ as our model set, it is convenient to refer to its metric properties directly.
So, we introduce the following definitions.
\begin{dfn}
For two points $a$ and $b$ of $\hilbert$ with $\enorm{a-b} \leq 2R,$ we call any shortest arc of a circle  of radius $R$ passing through $a$ and $b$ a \emph{main $R$-arc} of segment $[a,b].$ We will say that the length of a main $R$-arc
of segment $[a,b]$ is the \emph{$R$-arclength} of the segment $[a,b]$, that is,
the $R$-arclength of the segment $[a,b]$ is equal to $2R\arcsin{\frac{\enorm{a-b}}{2R}}.$
The $R$-arclength of a polyline $\Gamma$ is the sum of the $R$-arclengths of its segments provided that all of them are shorter than $2R.$ 
%In addition, the $R$-arclength of a segment whose length is strictly greater than $2R$ set to be $+\infty.$
\end{dfn}

\subsection{Curves, variations, and length}
\label{subsec:curves}
In the current subsection, we collect standard definitions related to curves and their length. To stress that the definitions are independent of the Euclidean structure, we provide definitions for Banach spaces.

% \begin{dfn}
We will say that a polyline $\Gamma$ is \emph{inscribed} in a set $A$ if all vertices of $\Gamma$ belong to $A.$ Additionally, if $A$ is a curve, we will assume the order of the vertices of $\Gamma$ is consistent with the parametrization of the curve. 
% \end{dfn}
% \begin{dfn}
For a polyline $\Gamma = x_0 \dots x_{n}$ in a normed space $(X, \norm{\cdot}),$ the \emph{mesh} of $\Gamma$  is equal to   $\max\limits_{i \in [n]}\norm{x_i-{x_{i-1}}}.$ We use $\mesh \Gamma$ to denote the mesh of $\Gamma.$ We will say that a polyline 
$x_0 \dots x_n$ is a \emph{refinement} of $y_0 \dots y_m$ if
$\{y_i\}_{i \in [m]}$ is a subsequence of $\{x_i\}_{i \in [n]}.$

We will say that a finite sequence $T= \{t_i\}_0^n$ with $ t_0 <  \dots < t_n $ is a \emph{partition} of the segment $[t_0, t_n] \subset \R.$
In the case $t_0 = t_n,$ we will say that the segment  $[t_0, t_n]$ is \emph{degenerate} and the singleton set $T = \{t_0\}$ is its partition.
Abusing the notation, we will use $\mesh T$ to denote $\max\limits_{i \in [n]} \abs{t_{i} - t_{i-1}}$ for a  partition  
 $T= \{t_i\}_0^n$  of a segment $[t_0, t_n]$ of the real line $\R$. We say that  a partition $T^\prime$ of a segment is a \emph{refinement} of partition $T$ if $T \subset T^\prime.$
 
By a curve we understand a continuous function from a non-degenerate segment of the real line to a normed space.
% \end{dfn}
 \begin{dfn}
     Fix a partition $T= \{t_i\}_0^n$  of a segment $[t_0,t_n].$
Let $f$ be a function on $[t_0, t_n]$ with values in a normed space $(X, \norm{\cdot}).$
      Denote
      \[
      \partvar[q]{f, T} = \parenth{\sum\limits_{k \in [n]} \norm{f(t_k) - f(t_{k-1})}^q (t_k - t_{k-1})^{1-q}  }^{1/q}.
      \]
 \end{dfn}
 In particular,
       \[
       \partvar[1]{f, T} = \sum\limits_{k \in [n]} \norm{f(t_k) - f(t_{k-1})}
       \quad \text{and} \quad 
      \partvar[2]{f, T} = \parenth{\sum\limits_{k \in [n]} \frac{\norm{f(t_k) - f(t_{k-1})}^2}
      {t_k - t_{k-1}} }^{1/2}.
      \]
  The following two statements are direct consequences of the 
  H\"older inequality.
     \begin{lem}\label{lem:summing_operators_add_points}
     Fix $q \geq 1.$
   Let  a partition $T^\prime $ of a segment $[t_0,t_1]$ be a refinement of another partition $T$  of the same segment.
  Let $f$ be a function on $[t_0,t_1]$ with values  in a normed space.
   Then 
   \[
   \partvar[q]{f, T} \leq \partvar[q]{f, T^\prime}.
   \]
  \end{lem}
   \begin{lem}\label{lem:summing_operators_comparison}
   Let $T= \{t_i\}_0^n$  with $0 = t_0 <  \dots < t_n = 1$ be a partition of $[0,1].$
  Let $f$ be a function on $[0,1]$ with values  in a normed space.
   Then 
   \[
   \partvar[q]{f, T} \leq \partvar[p]{f, T} 
   \]
   for all $1 \leq q \leq p.$
  \end{lem}
%   \begin{proof}
%      \[
%    \parenth{\partvar[q]{f, T}}^q = 
%   \sum\limits_{k \in [n]} \frac{\norm{f(t_k) - f(t_{k-1})}^q}{ (t_k - t_{k-1})^{q-1}} =
% \sum\limits_{k \in [n]} \frac{\norm{f(t_k) - f(t_{k-1})}^q}{ (t_k - t_{k-1})^{\frac{q}{p}(p-1)}} \cdot
%    (t_k - t_{k-1})^{1 - \frac{q}{p}}.
%    \]  
%  Using here the H\"older inequality for $ \frac{p}{q} \geq 1,$
%       \[
%    \parenth{\partvar[q]{f, T}}^q   \leq 
%     \parenth{\sum\limits_{k \in [n]} \frac{\norm{f(t_k) - f(t_{k-1})}^p}{ (t_k - t_{k-1})^{p-1}}}^{q/p}
%     \cdot \parenth{\sum\limits_{k \in [n]} (t_k - t_{k-1})}^{1 - q/p}  = \parenth{\partvar[p]{f, T}}^q.
%    \]
%   \end{proof}     
\begin{dfn}
  Let $f$ be a function on $[t_0,t_1]$ with values  in a normed space.
\emph{Variation of order} $p$ of $f$ is defined by
\[
\variation[p]{f}=\sup\limits_{T}\partvar[p]{f, T},
\]
where the supremum is taken over all partitions $T$  of $[t_0, t_1].$
\end{dfn}

% \begin{dfn}
For a continuous function $\gamma$  from  $[t_0,t_1]$ to a normed space, by the {\it length} of the curve $\gamma$ we understand  its first variation, that is, $\length{\gamma}=\variation{\gamma}.$
We say that a curve is \emph{rectifiable} if it has a finite length. 
\begin{dfn}
We say that a rectifiable curve $\gamma$ in a normed space is \emph{represented by its standard parametrization}, if $\gamma$ is represented as a function from the segment $[0,1]$  such that the length of the part of the curve between $\gamma(0)$ and $\gamma(t)$ is equal to $ t \cdot \length{\gamma}.$
\end{dfn}
Clearly, a rectifiable curve $\gamma$ represented by its standard parametrization is a  Lipschitz function with constant $\length{\gamma}.$ 
\begin{lem}\label{lem:variations_mesh_to_zero}
Fix $p \geq 1;$ and let $\gamma $ be a continuous curve from $[t_0, t_1]$ into a normed space.
For a sequence $\left\{T_i\right\}_{i \in \N}$ of partitions of the segment $[t_0, t_1]$   with
$\mesh{T_i} \to 0$ as $i \to \infty,$ 
\[
 \partvar[p]{\gamma, T_i} \to  \variation[p]{\gamma} \quad \text{as} \quad i \to \infty.
\]
\end{lem}
\begin{lem}\label{lem:second_variation_for_standard_parametrization}
Consider a rectifiable curve $\gamma$ from $[0,1]$ to a  normed space.  Then 
$
    \variation{\gamma} \leq \variation[2]{\gamma}.
$
If, in addition, $\gamma$ is represented by the standard parametrization, then 
$
    \variation{\gamma}= \variation[2]{\gamma}.
$
\end{lem}
\begin{proof}
    By \Href{Lemma}{lem:summing_operators_comparison} and 
    \Href{Lemma}{lem:variations_mesh_to_zero}, 
    $ \variation{\gamma} \leq \variation[2]{\gamma}.$ On the other hand, $\gamma$ is Lipschitz with constant $\variation{\gamma}$ if it is represented by the standard parametrization. Hence, $ \partvar[2]{\gamma, T} \leq \variation{\gamma}$ for any partition $T$ of $[0,1].$ The result follows.  
\end{proof}

\begin{dfn}
We define the distance $\operatorname{curve-dist} \! \parenth{\gamma_1, \gamma_2}$ between two rectifiable curves $\gamma_1$ and $\gamma_2$ given by their standard parametrizaion as $\variation{\gamma_1 - \gamma_2}.$
\end{dfn}

One can easily see that curve-dist obeys the axioms of metric for rectifiable curves starting at a fixed point. The following statement is a standard exercise on curves of bounded variation. However, we think it is better to sketch its proof since a function of bounded variation might be discontinuous. 
\begin{lem}\label{lem:convergences_in_variation}
The space of rectifiable curves in a Banach space endowed with metric $\operatorname{curve-dist}$ is complete in the following sense.
Consider  a fundamental in the sense of $\operatorname{curve-dist}$  sequence $\{\gamma_i\}_{i \in \N}$ of rectifiable curves  starting at a point $p$. Then, there exists a rectifiable curve 
$\gamma$ starting at $p$ of finite length to which the sequence $\{\gamma_i\}_{i \in \N}$ converges  in the sense of $\operatorname{curve-dist}$. 
\end{lem}
\begin{proof}
As $\{\gamma_i\}_{i \in \N}$ is fundamental in the sense of $\operatorname{curve-dist}$, its lengths converge to a fixed number $\ell$.
Fix a positive $\epsilon.$  There is a positive number $N_\epsilon$ such that for an arbitrary partition $T = \{t_q\}_0^n$ of the segment $[0,1]$ 
and all $i,j$ greater than $N_\epsilon$, inequalities
\[
 \partvar[1]{\gamma_i - \gamma_j, T} 
 \leq  \variation{\gamma_i - \gamma_j}  \leq\epsilon 
\quad \text{and} \quad 
\length{\gamma_i} \leq  \ell + \epsilon
\]
hold.   In particular, for all $t \in (0,1),$
\[
\norm{\gamma_i(t) - \gamma_j (t)} = 
\norm{\gamma_i(t) - \gamma_j (t) - \gamma_i(0) + \gamma_j (0)} \leq
 \partvar[1]{\gamma_i - \gamma_j, \{0, t, 1\}} \leq  \variation{\gamma_i - \gamma_j} \leq \epsilon.
\]
Similarly, $\norm{\gamma_i(1) - \gamma_j (1)} = 
 \partvar[1]{\gamma_i - \gamma_j, \{0,  1\}} \leq  \variation{\gamma_i - \gamma_j} \leq \epsilon.$
Thus, one can define $\gamma (t) $ as the limit of $\gamma_i(t)$ for all $t \in [0,1].$ 
 The rest of the proof is an appropriately chosen ``passing to the limit'' argument. 
Indeed, it follows that $\variation{\gamma - \gamma_j} \to 0$ as $j \to \infty.$
This  and the triangle inequality imply that $\gamma$ is Lipschitz with constant $\ell:$
\[
\norm{\gamma(t_1) - \gamma(t_2)} \leq 
\norm{\gamma(t_1) - \gamma_j(t_1)}  +
\norm{\gamma_j(t_1) - \gamma_j(t_2)} + 
\norm{\gamma_j(t_2) - \gamma(t_2)},
\]
which is less than 
$  ( \ell + \epsilon) \abs{t_1 -t_2} + 2\variation{\gamma - \gamma_j}.
$
Using the triangle inequality again, 
$
 \partvar[1]{\gamma, T}  \leq 
  \partvar[1]{\gamma_j, T} +  \partvar[1]{\gamma - \gamma_j, T} \leq
 \variation{\gamma_j} + \variation{\gamma - \gamma_j},
$
and 
$
\partvar[1]{\gamma, T} \geq   
\partvar[1]{\gamma_j, T} -  \partvar[1]{\gamma - \gamma_j, T}.
$
Passing to the limit as $\mesh T$ tends to zero and using 
\Href{Lemma}{lem:variations_mesh_to_zero}, we get that 
$ \variation{\gamma}  = \ell.$
\end{proof}

Again, by routine limiting procedure, one gets the following result.
\begin{lem}\label{lem:limit_of_polylines}
Let $a$ and $b$ be two distinct points of  a closed subset $A$ of a Banach space.
Let  $\left\{\gamma_i \right\}_{i \in \N}$ be a sequence  of polylines with endpoints $a$ and $b$ inscribed in $A$  such that $\mesh \gamma_i \to 0$ as $i \to \infty;$ for all natural $i,$ $\gamma_{i+1}$ is a refinement of $\gamma_i;$ and $\length{\gamma_i} \to \ell$ as $i \to \infty.$ Then there is a continuous curve $\gamma$ in $A$ connecting $a$ and $b$ of length $\ell$ in which all $\gamma_i$ are inscribed.
\end{lem}

\subsection{Intrinsic metric}

\begin{dfn}
 Let  $A$ be a subset of a normed space. 
The {\it  intrinsic metric} on $A$ is a semi-metric $\rho_A$ defined as follows: 
$\rho_A(x,y)$ is the infimum of the lengths of all curves  in $A$ connecting $x$  to  $y$ (setting $\rho_A(x,y) = \infty$ if there is no rectifiable curve in $A$ connecting the points).
\end{dfn}
In particular, \Href{Lemma}{lem:short_curve_intro} implies that the intrinsic distance between any sufficiently close points $a$
and $b$ of an $R$-proximally smooth subset of $\hilbert$ is less than   the intrinsic distance between $a$ and $b$ on the Euclidean sphere of radius $R$ passing through them. We believe it is a more natural formulation of the problem, which clearly relates the problem to the Alexandrov geometry \cite{gromov1999metric, burago2022course}.

The following observation shows that the intrinsic metric on a boundary of a convex set in a normed space coincides with the intrinsic metric on the complement of its interior. To put it differently,
 the infimum of lengths of curves connecting points $a$ and $b$ from a boundary of a convex set $K$ and lying outside of the interior of $K$ is at least the infimum of lengths of curves connecting points $a$ and $b$ and belonging to the boundary of $K.$ For the complete proof see \cite[Theorem 5B]{schaffer1976geometry}.
 We recall that a normed space is strictly convex if its unit sphere doesn't contain line segments.

\begin{lem}\label{lem:curve_separation}
Let $K$ be a convex set in a normed space  and let $a,b\in \partial K$. Consider a curve $\gamma$ from  $[0,1]$ to the complement of the interior of $K$ such that $\gamma(0)=a, $ $ \gamma(1)=b$. Then there is a curve $\Gamma:[0,1]\to\partial K$ such that $\Gamma(0)=a,$ $\Gamma(1)=b$ and $\length{\Gamma}\leq \length{\gamma}$. If the space is strictly convex and $\gamma$ contains a point outside $K$ the  inequality is strict.
\end{lem}

\begin{sketch} Here we just show a key observation of the proof of this result, which also implies the desired inequality in the case of a strictly convex space. 

Assume $\gamma$ contains point $c$ outside $K.$ Then $c$ and $K$ can be strictly separated by a hyperplane $H.$ By continuity, the intersection $\gamma \bigcap H$  contains at least two points such that $c$ belongs to the part of the curve between those points. 
Substituting that part of the curve by the segment $[e,f],$ we do not increase the length of the curve,
moreover, in the strictly convex case, we strictly decrease the length of the curve. \qed \end{sketch}

Note that for any two unit vectors $a$ and $b$ of $\hilbert,$ 
the intrinsic distance between them on the unit sphere 
$\sphere_{\hilbert}$ equals $\anglef{a}{b},$ which, in turn, is 
$
2 \arcsin{\frac{\enorm{a-b}}{2}}.
$
It follows that the angle satisfies the triangle inequality for non-zero vectors.

%We denote the intristic metric on its unit sphere  as $d_1$ and the i%ntrinsic metric on 
%$S \subset X$ as $d_S.$ 

\subsection{Proximally smooth sets and their properties}
\label{subsec:proximally_smooth_set}

% \begin{dfn}
%     The {\it metric projection} of a point $x $ onto a set $A$ is defined as any element of the set 
% $$
% 	P_A(x) = \{a \in A \st \enorm{a-x} = \dist{x}{A}\}.
% $$

% \end{dfn}

A  set $A$ in a Banach space is called $R$-proximally smooth if it is closed and  the distance function
 $x \mapsto \dist{x}{A}$ is continuously differentiable on 
 the  open $R$-neighborhood of $A.$

We say that a set $A$ in a Banach space satisfies the {\it supporting condition of $R$-weak convexity} if for any $x$ in the open $R$-neighborhood of $A$ and  any metric projection $p$ of $x$ onto $A,$
i.e. $\norm{x - p}= \dist{x}{A},$ inequality
$$
\dist{p + R\frac{x - p}{\norm{x - p}}}{A}\geq R
$$
holds.

% The following properties of the strongly convex segment were proven in paragraph 1.2 of \cite{IvMonEng}.

% \begin{prp}
%     Let $(X,\norm{\cdot})$ be a normed space, $x_0,x_1, y_0, y_1\in X$ such that $\norm{x_0-x_1}\leq 2R$ and $y_0, y_1 \in \strconvsegment{x_0}{x_1}$. Then $\strconvsegment{y_0}{y_1}\subset\strconvsegment{x_0}{x_1}$.
% \end{prp}

\begin{prp}\label{prp:equivalence}
Let $A$ be a closed set in $\hilbert.$ 
The following conditions are equivalent:
\begin{enumerate}
\item The set  $A$ is $R$-proximally smooth. 
\item The set $A$ satisfies the supporting condition of $R$-weak convexity.
\item The metric projection is singleton and continuous on the open $R$-neighborhood  of $A.$
\end{enumerate}
\end{prp}
In the  Hilbert space case, this proposition was proven already in \cite[Theorem 4.1]{Clarke1995}. We note that the same result holds in a more general setting of so-called uniformly convex and uniformly smooth Banach spaces \cite[Theorem 2.4]{IvBal2009}. In essence, this result is about computing the proximal gradient of a distance function in a Banach space \cite{borwein1987proximal}.

As a warm-up, we sketch the proof of the following simple statement that we will need in some technical estimates. It is simple and 
 its different versions are scattered among the papers on the weakly convex sets.

\begin{lem}\label{lem:projection_point_of_inscribed_segment}
   Let $a$ and $b$ be two points of an $R$-proximally smooth set $A$  with $\enorm{a-b} < 2R$. Then for any 
   $\lambda \in (0,1),$ the metric projection of the point $\lambda a + (1 - \lambda)b$ onto $A$ belongs to $\strconvsegment{a}{b}$ and
\[
\dist{\lambda a + (1 - \lambda)b}{A}\leq R-\sqrt{R^2-\lambda(1-\lambda)\enorm{a-b}^2}.
\]
\end{lem}
\begin{sketch}
    Denote $x = \lambda a + (1 - \lambda)b.$ The quantity in the right-hand side of the desired inequality is exactly the distance from $x$ to the boundary of $\strconvsegment{a}{b}.$ To see this, consider the power of $x$ with respect to any circle of  radius $R$ passing through 
$a$ and $b$ (the definition of the power of a point as well as a comprehensive treatment of this notion can be found in  \cite{coxeter1967geometry}). As the metric projection is continuous, the desired inequality will follow from the fact that the metric projection of $x$ onto $A$ exists and belongs to $\strconvsegment{a}{b}.$ 
The existence and uniqueness of metric projections follow from \Href{Proposition}{prp:equivalence} and the inequality 
$\min \left\{\enorm{a - x}, \enorm{b -x}\right\} \leq 
\frac{\enorm{a-b}}{2} < R.$ Let $p$ be the metric projection of $x$ onto $A.$ Suppose that $p$ is not in $\strconvsegment{a}{b}.$ 
Denote $o = p + R\frac{x-p}{\enorm{x-p}}.$ Consider the two-dimensional subspace passing through $a,b$ and $o$. Our assumption yields that $p$ and $o$ belong to two different half-planes with the boundary $ab.$ Thus, the disk of radius $R$ centered at $o$ contains at least one of the points $a$ and $b$
in its interior, which contradicts the supporting condition of $R$-weak convexity. Thus, 
$p$ belongs to $\strconvsegment{a}{b}.$ The proof is complete.
\qed \end{sketch}

\Href{Lemma}{lem:projection_point_of_inscribed_segment} says that a proximally smooth set $A$ contains a curve  connecting two of it sufficiently close points $a$ and $b;$ moreover,  such a  curve  exists in the set $A \cap \strconvsegment{a}{b}.$ Luckily, in the Euclidean case, the converse is also true \cite[Theorem 2.2]{IvBal2009}. 
\begin{prp}\label{prp:equivalence2}
Let $A$ be a closed set in $\hilbert.$ 
The following conditions are equivalent:
\begin{enumerate}
\item The set  $A$ is $R$-proximally smooth. 
\item For all  $a$ and $b$ in $A$ with
$0 < \enorm{a-b} < 2R,$ the intersection of $A$ and $\strconvsegment{a}{b} \setminus\{a,b\}$ is non-empty. 
\item For all  $a$ and $b$ in $A$ with
$0 < \enorm{a-b} < 2R,$ there  is a curve in the intersection of $A$ and $\strconvsegment{a}{b}$  connecting $a$ and $b.$ 
\end{enumerate}
\end{prp}

The following statement says that the metric projection onto an $R$-proximally smooth set is Lipschitz with constant
$\frac{R}{R- r}$ in the $r$-neighborhood of the set for some $r < R.$ In this particular form, the statement can be traced back to \cite[Theorem 4.8]{Clarke1995}. However, we need a more subtle argument to get  better estimates later. 
In the case of Banach spaces, the tight asymptotic  was obtained in \cite{Ivanov_2015}.

\begin{lem}\label{lem:metric_projection_onto_weakly_convex_set}
    Let $A$ be an $R$-proximally smooth subset of $\hilbert,$ 
    let $x_1, x_2$  be points such that
    $\dist{x_1}{A} < R$ and $\dist{x_2}{A} < R.$ 
Denote by $p_1$ and $p_2$ the metric projections of $x_1$ and $x_2$ onto $A,$ correspondingly.
Then
\[
\enorm{p_1 - p_2} \leq \frac{2R}{2R-(\dist{x_1}{A} + \dist{x_2}{A})} \enorm{x_1 - x_2}.
\]
\end{lem}
\begin{proof}
Assume $p_1 \neq x_1.$ Then, 
   by \Href{Proposition}{prp:equivalence},
$    \enorm{p_1 + \frac{R}{\enorm{x_1 - p_1} }(x_1 - p_1)  - p_2}^2 \geq R^2,
$
  or equivalently,
 $
 \enorm{p_1 - p_2}^2 \geq \frac{2R}{\enorm{x_1 - p_1} } \iprod{x_1 - p_1}{p_2 - p_1}.
 $ 
Hence,
  $
 \dist{x_1}{A} \enorm{p_1 - p_2}^2 \geq 2R \iprod{x_1 - p_1}{p_2 - p_1}.
 $ 
 In the case $x_1 = p_1,$ the latter inequality trivially holds.
The same arguments work for $x_2$ and $p_2,$ i.e. 
\[
   \dist{x_1}{A} \enorm{p_1 - p_2}^2 \geq 2R \iprod{x_1 - p_1}{p_2 - p_1} 
  \quad \text{and} \quad 
   \dist{x_2}{A} \enorm{p_1 - p_2}^2 \geq 2R \iprod{x_2 - p_2}{p_1 - p_2}.
\]
Summing up, we have
\[
(\dist{x_1}{A} + \dist{x_2}{A}) \enorm{p_1 - p_2}^2 \geq 
2R \enorm{p_1 - p_2}^2 + 2R \iprod{x_2 - x_1}{p_1 - p_2}.
\]
Hence, 
\[
(\dist{x_1}{A} + \dist{x_2}{A}) \enorm{p_1 - p_2}^2 \geq
2R \enorm{p_1 - p_2}^2 - 2R \enorm{x_2 - x_1} \cdot \enorm{p_1 - p_2}. 
\]
The desired inequality follows.
\end{proof}

\subsection{Slice-projection}
\begin{dfn}
Let $x_1$ and $x_2$ be two distinct points of $\hilbert,$ 
and let $H$ be a hyperplane orthogonal to $x_2 -x_1$ and intersecting the segment
$[x_1, x_2]$ at point $x.$ Any point of the metric projection of $x$ onto the set 
$
 A \cap H
$ we call a \textit{slice-projection}  of $x$ onto a set $A$ w.r.t.  $[x_1,x_2].$
  \end{dfn}

  \begin{lem}\label{lem:uniqueness_slice_projection}
      Let $A$ be  $R$-proximally smooth,
let $x_1$ and $x_2$  be two distinct points of $A$ with $ {\enorm{x_1 - x_2}} < 2R.$ Then for any 
$x_t=t x_1+(1-t)x_2,$ $t\in[0,1]$, there is a unique slice-projection of $x_t$ onto $A$ with respect to $[x_1, x_2];$ the slice-projection belongs to $D_R[x_1, x_2].$
\end{lem}
\begin{proof}
 Fix any $t\in [0,1]$.  Let $H$ be the hyperplane passing through $x_t$ and orthogonal to $x_2-x_1$. For every point on the segment $[x_1, x_2],$ its metric projection  onto $A$ is singleton. 
By \Href{Lemma}{lem:metric_projection_onto_weakly_convex_set}, the metric projection of the segment 
$[x_1, x_2]$ is a continuous curve. It intersects the hyperplane $H$ at a point $c.$ Let $c$ be a metric projection of the point $x$ from the segment
$[x_1, x_2].$ We claim that $c$ is the unique slice-projection of $x_t.$
Indeed, for any point $c^\prime$ in the intersection of $H$ and the ball of radius $\enorm{c - x_t}$ centered at $x_t,$ one has
\[
\enorm{x - c^\prime}^2 = \enorm{x - x_t}^2 + \enorm{x_t - c^\prime}^2 \leq 
\enorm{x - x_t}^2 + \enorm{x_t - c}^2 = \enorm{x-c}^2.
\]
The lemma follows since the metric projection of $x$  onto $A$ is a singleton.
\end{proof}

\section{Strongly convex segment and its properties}
\label{sec:strongly_convex_segment}
The following two statements are trivial.
\begin{clm}
Let $a$ and $b$ be points of $\hilbert$ with $0 < \enorm{a-b} \leq 2R.$ Then the boundary of $\strconvsegment{a}{b}$ is the union of all $R$-arcs between $a$ and $b.$  
\end{clm}
 \begin{lem}\label{lem:angle_inside_strongly_convex_segment}
 Let $a$ and $b$ be points of $\hilbert$ with $0 < \enorm{a-b} \leq 2R.$ Let  $\ell$ be a ray  emanating from $a.$ Then  the intersection of $\ell$ and $\strconvsegment{a}{b}\setminus\{a,b\}$ is not-empty if and only if  $\anglef{\ell}{b - a} < \arcsin{\frac{\enorm{b-a}}{2R}}.$ 
 % Equivalently, 
 % $c \in \\strconvsegment{a}{b}\setminus\{a,b\}$  if and only if $\anglef{a-c}{b-c} \geq \pi - \arcsin{\frac{\enorm{b-a}}{2R}}.$
 \end{lem}
 % \begin{proof}
 %     Trivially follows from the previous claim.
 % \end{proof}
% \begin{clm}
% Let $\gamma$ be a polyline  $abc$ with $\enorm{a-c} \leq 2R.$
% \begin{itemize}
%     \item If $b$ is in the interior of $\strconvsegment{a}{c},$ then the $R$-arclength of $\gamma$ is strictly less than the $R$-arclength of $[a,c];$
%      \item If $b$ belongs to the boundary of $\strconvsegment{a}{c},$ then the $R$-arclength of $\gamma$ is equal to the $R$-arclength of $[a,c];$
%        \item If $b$ lies outside of $\strconvsegment{a}{c},$ then the $R$-arclength of $\gamma$ is strictly greater than the $R$-arclength of $[a,c].$
% \end{itemize}
% \end{clm}
\begin{lem}\label{lem:radial_projection_not_extending}
Let $x_1$ and $x_2$ be two points of $\hilbert$ on the sphere of radius $R$ centered at a point $o.$
Let $y_1$ and $y_2$ be points such that $x_1$ belongs to the segment $[o, y_2]$ and  
 $x_2$ belongs to the segment $[o, y_1]$. Then $\enorm{y_1 - y_2} \geq \enorm{x_1 - x_2}$
 with equality if and only if $x_1 =y_1$ and $x_2 = y_2.$
\end{lem}
\begin{proof}
The statement is trivial in the cases  when  $x_1, o, x_2$ belong to one line.
Assume the points  $x_1, o, x_2$ form a triangle. By positive homogeneity and symmetry, we can assume that $y_1$ coincides with $x_1.$ In this case the angle $x_1 x_2 y_2$ is obtuse, since the triangle $x_1ox_2$ is isosceles.
The lemma follows.
\end{proof}
\begin{dfn}\label{dfn:folding mapping}
  Let $\ell$ be a line in a Hilbert space $\hilbert,$ and let $u$ be a unit vector orthogonal to $\ell.$
  Define $L_2$ as the two-dimensional half-plane with relative boundary $\ell$ and containing a vector collinear with $u.$ We define a \emph{folding mapping} $P_{L_2} \colon \hilbert \to L_2$ 
  as follows. Fix a unit vector $u$ parallel to  $L_2$ and orthogonal to $\ell.$
  For any $x$ of the form $l + p,$ where $l \in \ell$ and $p \in \ell^\perp,$
  we define $P_{L_2}(x) = l + \enorm{p}u.$
\end{dfn}
\begin{clm}\label{claim:folding_mapping_is_not_extending}
In the notations of \Href{Definition}{dfn:folding mapping},
$\enorm{P_{L_2}(x) - P_{L_2}(y)} \leq \enorm{x - y}$ for any $x,y \in \hilbert.$
The identity holds if and only if $x$ and $y$ belong to a half-plane whose relative boundary is the line $\ell.$
\end{clm}
\begin{proof}
    Assume $x_1 = l_1 + p_1$ and   $x_2 = l_2 + p_2,
    $ where $l_1, l_2 \in \ell$ and $p_1, p_2 \in \ell^\perp.$
    Then by Pythagorean theorem and the triangle inequality,
    \[
    \enorm{x_1 - x_2}^2 = \enorm{l_1 - l_2}^2 + \enorm{p_1 - p_2}^2 \geq 
    \enorm{l_1 - l_2}^2 + \abs{\enorm{p_1} -\enorm{p_2}}^2 = \enorm{P_{L_2}(x_1) - P_{L_2}(x_2)}^2.
    \]
\end{proof}

\begin{lem}\label{lem:arclength}
    Let $a$ and $b$ be points of $\hilbert$ with $\enorm{a-b} < 2R.$
    Let $\gamma = a x_1 \dots x_n b$ be a polyline connecting $a$ and $b,$ and such that all its vertices 
    lie outside the interior of $\strconvsegment{a}{b}.$ Then the $R$-arclength of $\gamma$ is  greater than or equal to  the $R$-arclength of $[a,b].$ The identity holds if and only if $x_1, \dots, x_n$ are consecutive points of a main $R$-arc of     $[a,b].$
\end{lem}
\begin{proof}
Fix an arbitrary two-dimensional half-plane $L_2$ with the relative boundary
$ab$ spanned by  $b-a$ and a unit vector $u$ orthogonal to $b-a.$   Now we fold $\gamma$ into $L_2$ using the above-defined folding mapping $P_{L_2}.$ Denote $\gamma_2 = aP_{L_2}(x_1)...P_{L_2}{(x_n)}b.$ By \Href{Claim}{claim:folding_mapping_is_not_extending}, each segment of 
$\gamma_2$ is not longer than the corresponding segment of $\gamma.$ Consequently, the  
$R$-arclength of $\gamma_2$ is not greater than that of $\gamma,$ and the identity holds if and only if $\gamma$ belongs to a two-dimensional half-plane. However, $\gamma_2$ is contained in $L_2$.
Let $o$ be  the center of the circle $C$ of radius $R$ passing through $a$ and $b$ and lying in the half-plane spanned by $b-a$ and $-u$.  Denote  by $y_i$ the radial projection onto $C$ of the point $P_{L_2}(x_i),$ $i \in [n].$ By 
\Href{Lemma}{lem:radial_projection_not_extending}, the $R$-arclength of $a y_1 \dots y_n b$ is at most that of
$\gamma_2.$ Clearly, the $R$-arclength of $a y_1 \dots y_n b$ is at least the $R$-arclength of a main $R$-arc of segment $[a,b].$ The inequality and the identity case follow. 
\end{proof}

The following claim is trivial.
\begin{clm}\label{clm:slice_projection}
    Let $a,b\in \hilbert$ be such that $0<\enorm{a-b}<2R$. Let $H_{a - b}$ be a 
    hyperplane orthogonal to $[a,b]$ and 
    $H_{a - b}\bigcap [a,b]\neq\emptyset$. 
    Denote by $c$ a point from 
    $H_{a - b}\bigcap\partial \strconvsegment{a}{b}$. 
    Then for any $c'\in H_{a - b}\bigcap \strconvsegment{a}{b}$ 
    one has $\enorm{a-c'}\leq\enorm{a-c}$ and
    $\enorm{b-c'}\leq\enorm{b-c}$. 
    The identity is attained if and only if $c'\in\partial \strconvsegment{a}{b}$.
\end{clm}

\begin{lem}\label{lem:arc}
    Assume $0 < \enorm{a-b}< 2R$. 
    Then the length of any continuous curve $\gamma$ connecting $a$ and $b$ from the boundary of 
    $\strconvsegment{a}{b}$ is at least the $R$-arclength of segment $[a,b].$ Moreover, the identity holds if and only if 
    $\gamma$ is a main $R$-arc of $[a,b].$
%    Let $\gamma:[0,1]\to \partial \strconvsegment{a}{b}$ with $\gamma(0)=a$, $\gamma(1)=b$. Then  $\length{\gamma}\ge
%    2R\arcsin{\frac{\enorm{a-b}}{2R}}$.The identity is obtained if and only if
%    $\gamma$ is a main arc.
\end{lem}
\begin{proof} Fix an arbitrary two-dimensional half-plane $L_2$ passing through the line
$ab$ and a point $c$ of $\gamma$ different from $a$ and $b$.  Let $u \parallel L_2$ be a unit vector orthogonal to $b-a$. Now we fold $\gamma$ into $L_2$ using the above-defined folding mapping $P_{L_2}.$ 
By \Href{Claim}{claim:folding_mapping_is_not_extending}, for any polyline $ax_1...x_nb,$ we have that the polyline 
$aP_{L_2}(x_1)...P_{L_2}{(x_n)}b$ is not longer, and its mesh
is less than or equal to that of $ax_1...x_nb$. 
The desired inequality follows from \Href{Lemma}{lem:variations_mesh_to_zero}. 

Assume that the length of $\gamma$ is equal to the $R$-arclength of segment $[a,b].$
Denote $\omega_1 =\partial \strconvsegment{a}{b} \cap \partial \strconvsegment{a}{c}$ and $\omega_2 = \partial \strconvsegment{a}{b} \cap \partial \strconvsegment{c}{b}.$ Clearly, $\omega_1$ and $\omega_2$ are main $R$-arcs.
Thus, by the already proven inequality of the lemma, the length of the part of $\gamma$ between $a$ and $c$ is at least the length of $\omega_1,$  and the length of the part of $\gamma$ between $c$ and $b$ is at least the length of $\omega_2.$ Since $\length{\omega_1} + \length{\omega_2} = 2R\arcsin{\frac{\enorm{a-b}}{2R}},$  we conclude that the identity holds in both previous inequalities. 
However, using the observation that the interior of $\strconvsegment{a}{c}$ belongs to the interior $\strconvsegment{a}{b}$ in \Href{Lemma}{lem:curve_separation} and the definition of $\omega_1$ and $\omega_2$, we conclude that $\gamma$ coincides with the union of $\omega_1$ and
$\omega_2.$ The proof is complete.
\end{proof}
\begin{rem}
   Note that \Href{Lemma}{lem:arc} is the limit version of \Href{Lemma}{lem:arclength}.
\end{rem}
\Href{Lemma}{lem:arc} together with \Href{Lemma}{lem:curve_separation} yield the following.
\begin{cor}\label{cor:main_arc}
        Assume $0 < \enorm{a-b}< 2R$. 
    Then the length of any continuous curve $\gamma$ connecting $a$ and $b$ lying outside of the interior of $\strconvsegment{a}{b}$ is at least the $R$-arclength of segment $[a,b].$ Moreover, the identity holds if and only if      $\gamma$ is a main $R$-arc of $[a,b].$
\end{cor}

We will need a quantitative version of this corollary to obtain which we prove the following quantitative version of the triangle inequality.
\begin{clm}\label{clm:epsilon_aptitude_triangle_inequality}
Fix $\epsilon \in [0,1].$
    Let $x$ and $y$ two distinct points of $\hilbert.$ Let $z$ be such that the distance from $z$
    to the line $xy$ is at least $\epsilon.$ Then 
    \[
    \enorm{x-z} + \enorm{z - y} \geq \enorm{x-y} + 
    \min\left\{\frac{\epsilon^2}{2 \enorm{x-y}}, \ \frac{\epsilon}{2}\right\}.
    \]
\end{clm}
\begin{proof}
    Denote the metric projection of $z$ onto the line $xy$ as
    $p.$ If $p \notin [x,y],$ then 
    \[ 
    \min\{\enorm{x-z}, \enorm{z - y}\} \geq \epsilon \quad \text{and} \quad
  \max\{\enorm{x-z}, \enorm{z - y}\} \geq \sqrt{\enorm{x-y}^2 + \epsilon^2} \geq \enorm{x - y}.\]
   The claim follows in this case.
    Assume $p \in [x,y].$ Clearly, 
    $\enorm{x - z} + \enorm{z - y} - \enorm{x - y} =
      \parenth{ \enorm{x- z} - \enorm{x-p}} +
      \parenth{\enorm{y-z} - \enorm{y-p}}.$ Let us estimate the differences in the parentheses. Due to symmetry, it suffices to consider only one of them: 
    \[
    \enorm{x - z} - \enorm{x- p} = 
    \frac{\enorm{x - z}^2 -\enorm{x- p}^2}{\enorm{x - z} + \enorm{x- p}}=
    \frac{\enorm{z - p}^2}{\enorm{x - z} + \enorm{x- p}} \geq 
     \frac{\epsilon^2}{\enorm{x - z} + \enorm{x- p}}.
    \]
     Consider the case 
    $\enorm{x-y}  \leq \epsilon.$
    Then $\enorm{x- p} \leq \enorm{x-y} \leq \epsilon$ and
    $\enorm{x -z} = \sqrt{\enorm{x-p}^2 + \enorm{p-z}^2} \leq \epsilon \sqrt{2}.$ That is, 
\[
\enorm{x - z} - \enorm{x- p} \geq \frac{\epsilon}{\sqrt{2}+1}.
\]    
Using the same bound for the second difference, one gets 
    \[
    \enorm{x-z} + \enorm{z - y} \geq 
    \enorm{x-p}  + \frac{\epsilon}{\sqrt{2} + 1} + \enorm{y-p} + \frac{\epsilon}{\sqrt{2} + 1} \geq 
    \enorm{x-y}  + \frac{\epsilon}{2}.
    \]
   Consider the case $\enorm{x-y}  > \epsilon.$ 
  Since $\enorm{x-p} \leq \enorm{x- y},$ we have   
\[
    \enorm{x- z} = \sqrt{\enorm{x-p}^2 + \enorm{p-z}^2} \leq
    \enorm{x-y} \sqrt{1 + \frac{\epsilon^2}{\enorm{x-y}^2}} \leq \sqrt{2} \enorm{x- y}.
\]
That is, 
\[
\enorm{x - z} - \enorm{x- p}
 \geq \frac{\epsilon^2}{\parenth{\sqrt{2}+1}\enorm{x-y}}.
\]    
Using the same bound for the second difference, one gets 
\[
    \enorm{x-z} + \enorm{z - y} \geq 
    \enorm{x-p}   + \enorm{y-p} + 2 \cdot \frac{\epsilon^2}{\parenth{\sqrt{2}+1}\enorm{x-y}} \geq 
    \enorm{x-y}  + \frac{\epsilon^2}{2\enorm{x-y}}.
\]
\end{proof}

\begin{lem}\label{lem:curve_separation_stongly_convex_segment}
      Assume $0 < \enorm{a-b}< 2R$. 
   Assume  curve $\gamma$ connecting $a$ and $b$ lies  outside of the interior of $\strconvsegment{a}{b}$  and contains a point $z$ at a distance $\delta$ from $\strconvsegment{a}{b}.$
   Then
   \[
     \length{\gamma} \geq \min\left\{ 2R \arcsin \frac{\enorm{a-b}}{2R} +   \min\left\{\frac{\delta^2}{2 \pi R}, \ \frac{\delta}{2}\right\}, \ \pi R \right\}.
   \]
\end{lem}
\begin{proof}
We  use $q$ to denote the metric projection of $z$ onto $\strconvsegment{a}{b}.$
Let $H$ be the supporting hyperplane to $\strconvsegment{a}{b}$ orthogonal to 
$z -q.$ The hyperplane $H$ separates $z$ from $\strconvsegment{a}{b}.$ Thus,
$\gamma$ intersects $H$ at least twice, say at $x$ and $y.$ Without loss of generality, we consider that $z$ is between $x$ and $y$.
If $\enorm{x-y} \geq \pi R,$ then $\length{\gamma} \geq \pi R.$
Assume $\enorm{x-y} \leq \pi R.$
The distance from $z$ to the line  $xy$ is at least that from $z$ to $H,$ that is, $\delta.$ Substitute the part of $\gamma$ between $x$ and $y$ by the segment $[x,y].$ By  \Href{Claim}{clm:epsilon_aptitude_triangle_inequality}, the  new curve is shorter than $\gamma$ by $ \min\left\{\frac{\delta^2}{2 \enorm{x-y}}, \ \frac{\delta}{2}\right\},$ which is at least  
 $ \min\left\{\frac{\delta^2}{2\pi R }, \ \frac{\delta}{2}\right\}.$  By  \Href{Corollary}{cor:main_arc} its length is at least the $R$-arclength of $[a,b].$ The lemma follows.
\end{proof}

\section{Short curves in proximally smooth sets}
\label{sec:short_curves}
In this section, we prove several simple lemmas about curves in proximally smooth sets and strongly convex segments. Our main tool here is the folding mapping. It will allow us to ``fold'' and 
``unfold'' curves while keeping the required inequalities on their  lengths.
% \begin{lem}\label{lem:smallness}
%     Let $\gamma_0 = x_1 \dots x_n$ be a polyline inscribed in a main arc of 
%     $\strconvsegment{x_1}{x_n}$. Assume $y_1, y_n$ belong to a proximally smooth 
%     set $A$ with constant $R$ and $0 < \enorm {y_1-y_n}\leq \enorm{x_1-x_n}$. 
%     Then there is a polyline $\gamma=y_1 \dots y_n$ inscribed in $A \cap \strconvsegment{y_1}{y_n}$  such
%     that $\enorm{y_i-y_{i+1}} \leq \enorm{x_i-x_{i+1}}$ for all $i \in [n-1].$ 
% \end{lem}
% \begin{proof} We use the induction on $n$. It is nothing to prove for $n=1$. 
% Assume the statement was proven for $n=k$. Let us prove it for $n=k+1$.
%  If $\enorm{x_1-x_2}\geq \enorm {y_1-y_n}$, we take  $y_2=...=y_n$. 

%  Consider the case $\enorm{x_1-x_2} < \enorm{y_1-y_n}.$
%  Let $H$ be a hyperplane such that $H$ is orthogonal to $y_1-y_n$ and 
%  $\dist{y_1}{H}\bigcap \partial \strconvsegment{y_1}{y_n} = \enorm{x_1-x_2}$.

%  Since $A$ is proximally smooth with constant $R$, there is a point $y_2\in H\bigcap \strconvsegment{y_1}{y_n}$. By \Href{Claim}{clm:slice_projection}, $\enorm{y_1-y_2}\leq \enorm{x_1-x_2}$ and 
%  $\enorm{y_2-y_n} \leq \enorm{x_2-x_n}$. Thus, by the induction hypothesis, there exist points $y_2, \dots, y_{n-1}$ in $A$ satisfying the desired inequalities. The lemma is proven. 
%  \end{proof}

 In the following simple algorithm, we show how to construct a polyline inscribed in a proximally smooth set whose $R$-arclength is at most the $R$-arclength between the segment formed by its endpoints.  
 \medskip
 
\begin{algorithm}[H]\label{algo1}
\label{algorithm}
\SetAlgoLined 
\KwData{A polyline $\gamma_0 = x_0 \dots x_n$ with $\enorm{x_0 - x_n} < 2R$ inscribed in a main $R$-arc of     $[x_0,x_n]$.  An  $R$-proximally smooth set $A$, 
two points $y_0$ and $y_n$ in $A$ satisfying  inequality $ \enorm{y_0 -y_n} \leq \enorm{x_0 -x_n}.$
}
\KwResult{ polyline $\gamma=y_0 \dots y_n$ inscribed in $A \cap \strconvsegment{y_0}{y_n}$  such
    that $\enorm{y_{i}-y_{i-1}} \leq \enorm{x_i-x_{i-1}}$ for all $i \in [n].$} 
   Set $i =0.$ 
\begin{enumerate}
\item \label{alg1_step_1} Stop if $i = n.$
\item If the inequality $\enorm{y_i - y_n} \leq \enorm{x_i - x_{i+1}}$ holds, 
set  $y_{i+1} = \dots =y_{n-1} = y_n$ and stop. 
\item If  the identity $x_i = x_{i+1}$ holds, set $y_{i+1} = y_i$, increment $i,$ and return to step  \eqref{alg1_step_1}.
\item  Let $H$ be a hyperplane such that $H$ is orthogonal to $y_i-y_n$  and 
$\dist{y_i}{ H \cap \partial \strconvsegment{y_i}{y_n}} = \enorm{x_{i+1} - x_{i}}.$
Set $y_{i+1}$ to be a  point of the intersection 
$H \cap \strconvsegment{y_i}{y_n}\cap A.$
Increment $i$ and return to step  \eqref{alg1_step_1}.
 \end{enumerate} 
  \caption{Unfolding of a flat curve}
\end{algorithm}
\begin{lem}
\Href{Algorithm}{algo1} is correct. It returns a polyline whose $R$-arclength is at most
the $R$-arclength of the segment $[y_0,y_n].$
\end{lem}
\begin{proof}
We use the induction on $n$. It is nothing to prove for $n=1$. 
Assume the statement was proven for $n=k$. Let us prove it for $n=k+1$.

If either $x_0=x_1$ or $\enorm{y_0 - y_n} \leq \enorm{x_0 - x_1},$  the induction hypothesis  yields the result. 

So, assume that $0 < \enorm{x_0 - x_1} < \enorm{y_0 - y_n}.$ Let $H$ be a hyperplane such that $H$ is orthogonal to $y_0-y_n$ and 
\[
\dist{y_0}{H\bigcap \partial \strconvsegment{y_0}{y_n}} = \enorm{x_0-x_1}.
\]
By \Href{Lemma}{lem:uniqueness_slice_projection}, the intersection $H \cap A \cap \strconvsegment{y_0}{y_n}$ is non-empty. 
By the hereditary property of the strongly convex segment (see \Href{Remark}{rem:str_convex_segment}), $\strconvsegment{y_1}{y_n}$ is a subset of
$\strconvsegment{y_0}{y_n}$ for any $y_1$ from the intersection $H\bigcap \partial \strconvsegment{y_0}{y_n}.$ 
By  \Href{Claim}{clm:slice_projection},   
$\enorm{y_0-y_{1}}\leq \enorm{x_0-x_{1}}$ and 
 $\enorm{y_{1}- y_n} \leq \enorm{x_{1}-x_n}$. 
 These inequalities also imply that the $R$-arclength of $y_0y_1y_n$ is at most that of the polyline $x_0x_1x_n.$ Thus,  the lemma follows from the induction hypothesis.
\end{proof}

We note that in \Href{Algorithm}{algo1}, one can choose the points $y_{i+1}$ in the fourth step greedily so that it is the slice-projection of the corresponding point of the segment $[y_i, y_n].$ In fact, it was shown in the previous paper of the first two authors \cite{ivanov2022rectifiable} that such a choice leads to a reasonably short curve in a proximally smooth subset of a uniformly convex and uniformly smooth Banach space.  Now, we can say that our intuition was right: In the Euclidean case, our algorithm returns the best possible curve in the sense of the upper bound on the length. The bound is clearly tight for $R \sphere_{\hilbert}.$

The only difference between the following algorithm and the original one from  \cite{ivanov2022rectifiable}  is the precise bound
on $\enorm{x_0 - x_1}$ and the bound on length on the curve.

\begin{algorithm}[H]
\label{algo2}
\SetAlgoLined 
\KwData{ An  $R$-proximally smooth $R$ set $A \subset \hilbert$, 
two points $x_0$ and $x_1$ in $A$ satisfying  the inequality $ \enorm{x_0 -x_1} < 2R.$}
\KwResult{ A rectifiable curve $f \colon [0,1] \to A \cap \strconvsegment{x_0}{x_1}$ with 
$f(0) = x_0$ and $f(1) = x_1$ of length at most
$2R \arcsin \frac{\enorm{x_0 - x_1}}{2R}.$}
Set $S_0 = \{0,1\}$ and $S_i = \{\frac{j}{2^i} \mid j \in [2^i]\}\} \cup \{0\}$ for $i \in \N.$
\begin{enumerate}
 \item Define $f$ at points of $S_0$ as follows: $f(0) = x_0$ and $f(1) = x_1.$
 \item  For every $i \in \N,$ we extend the domain of $f$ to the set $S_i \setminus S_{i-1}$ as follows:
 
set the value of $f$ at $\frac{2j-1}{2^i}$ to be the slice-projection of the midpoint  of the segment 

$\left[f\!\left(\frac{j-1}{2^{i-1}}\right), f\!\left(\frac{j}{2^{i-1}}\right)\right]$ for all $j \in [2^{i-1}]$ on $A$ with respect to the segment.
\item Continuously extend $f$ on $[0,1].$ 
 \end{enumerate} 
  \caption{Slice-bisection}
\end{algorithm}

\begin{lem}\label{lem:algo2_correctness}
\Href{Algorithm}{algo2} is correct. 
\end{lem}
\begin{proof}
The restriction of $f$ onto the set $S_{i}$ is a polyline, say $\Gamma_{i}.$ By \Href{Lemma}{lem:uniqueness_slice_projection}, we conclude that
$ f\!\!\left(\frac{2j-1}{2^{i}}\right)  \in
\strconvsegment{f\!\! \left(\frac{j-1}{2^{i-1}}\right)}{ f\!\! \left(\frac{j}{2^{i-1}}\right)}.$ By the hereditary property of strongly convex segments (see \Href{Remark}{rem:str_convex_segment}),
$\Gamma_i$ are inscribed in $A \cap \strconvsegment{x_0}{x_1}.$
Using the  bound from \Href{Lemma}{lem:projection_point_of_inscribed_segment} with $\lambda = \frac{1}{2}$ in \Href{Lemma}{lem:uniqueness_slice_projection},  we get $\mesh{\Gamma_{i+1}} \leq \frac{\mesh \Gamma_{i}}{\sqrt{2}}.$
In particular, $\mesh \Gamma_i \to 0$ as $i \to \infty.$ Clearly,
$\Gamma_{i+1}$ is a refinement of $\Gamma_i.$ However, the $R$-arclength of $\Gamma_{i+1}$ is not greater than that of $\Gamma_i$ by \Href{Claim}{clm:slice_projection}. Thus, the monotonically increasing sequence of the lengths of $\Gamma_i$ has a finite limit. The use of \Href{Lemma}{lem:limit_of_polylines} completes the proof.
\end{proof}

\subsection{Proof of \Href{Lemma}{lem:short_curve_intro} and related results}

The following result is a generalization of  \Href{Lemma}{lem:short_curve_intro}.
 \begin{lem}\label{lem:simple_short_curve_in_prox_smooth_set}
    Let $A\subset \hilbert$ be an $R$-proximally smooth set containing two distinct points $a$ and $b$ with    
    $\enorm{a-b}<2R.$ 
    Then there is a curve $\Gamma_{a,b}$ in $A \cap \strconvsegment{a}{b}$ connecting $a$ and $b,$ and
    whose length is at most the $R$-arclength of segment $[a,b],$ i.e. $\length{\Gamma_{a,b}}\leq 2R\arcsin{\frac{\enorm{a-b}}{2R}}.$ 
    In particular, $\rho_A(a,b) \leq 2R\arcsin{\frac{\enorm{a-b}}{2R}}.$
 \end{lem}
\begin{proof}
    Directly follows from the correctness of \Href{Algorithm}{algo2} proven in \Href{Lemma}{lem:algo2_correctness}.
\end{proof}
\begin{rem}
\Href{Algorithm}{algo1} provides a more general approach to constructing a short curve in a proximally smooth set. For example, one may proceed as follows.
 Fix $\e\in (0,1)$.
 Consider a polyline $\gamma_1$ inscribed in a main arc of $[a,b]$ of mesh $\e$.
 Using \Href{Algorithm}{algo1}, we construct a polyline $\Gamma_1$ inscribed in $A$ such that $\Gamma_1(0)=a$, $\Gamma_1(1)=b$ and the length of the $i$-th segment of $\Gamma_1$ is less than or equal to the  length of the $i$-th segment of $\gamma_1$. In particular, this implies that 
the $R$-arclength of $\Gamma_1$ is less than that of $\gamma_1,$ which is equal to the
$R$-arclength of segment $[a,b].$

Now we briefly describe the process of how to construct a new  polyline $\Gamma_{i+1}$ from $\Gamma_i$, such that
$\mesh \Gamma_{i+1}\leq\e^{i+1}$ and  its $R$-arclength is less than or equal to  that of $\Gamma_i.$

Take a non-trivial segment $[x,y]$ of $\Gamma_i$ of length $l$. Consider a polyline $\gamma$ inscribed in a main $R$-arc of length $l$ of mesh $\e^{i+1}$.  Using  \Href{Algorithm}{algo1},  we get a polyline $\Gamma$ inscribed in $A$ such that $\Gamma(0)=x$, $\Gamma(1)=y$ and the $i$-th segment of
 $\Gamma$ is  shorter than the $i$-th segment of $\gamma.$ 
We obtain the polyline $\Gamma_{i+1}$ by concatenating all such $\Gamma$ for all consecutive non-trivial segments $[x,y]$ of $\Gamma_i.$
\end{rem}

\begin{lem}\label{lem:shorter_curve_inside_str_convex_segment}
     Let $A\subset \hilbert$ be a closed, $R$-proximally smooth set, and let $\gamma$ be a curve
  in $A$ connecting $a$ and $b$ with $0 < \enorm{a-b} \leq 2R.$  If $\gamma \not \subset \strconvsegment{a}{b},$
  there exists a strictly shorter curve $\Gamma$ in $A \cap \strconvsegment{a}{b}$ connecting $a$ and $b.$
\end{lem}
\begin{proof}
The statement follows  from \Href{Lemma}{lem:simple_short_curve_in_prox_smooth_set} in the case when the length of $\gamma$ is strictly greater than the $R$-arclength of the segment $[a,b].$ 

Assume that $\gamma$ is rectifiable. 
    Consider  points $a_1$ and $b_1$ on $\partial \strconvsegment{a}{b} \cap \gamma$ such that the part of $\gamma$
    between them lies outside of $\strconvsegment{a}{b}.$  Then this part of the curve lies outside $\strconvsegment{a_1}{b_1}$
    as well. By  \Href{Lemma}{lem:simple_short_curve_in_prox_smooth_set} and \Href{Corollary}{cor:main_arc}, this part of $\gamma$
    can be substituted by a strictly shorter  curve from $A \cap \strconvsegment{a_1}{b_1},$ which is a subset of 
    $A \cap \strconvsegment{a}{b}.$ By the continuity, there are at most countably many pairs of such points $a_1, b_1.$ The lemma follows. 
\end{proof}

% \begin{thm}\label{thm:weakly_convex_equivalent_existence_short_curve}
%     Let $A\subset \hilbert$ be a closed set. The following assertions are equivalent:
% \begin{enumerate}
%     \item $A$ is $R$-proximally smooth;
%     \item  For any $a,b\in A$ such that $0<\enorm{a-b}<2R,$ 
%     there is a curve $\Gamma$ in $A$ connecting $a$ and $b,$ and
%     whose length is at most the $R$-arclength of segment $ab,$ i.e. $\length{\Gamma}\leq 2R\arcsin{\frac{\enorm{a-b}}{2R}}.$
%     \item For any $a,b\in A$ such that $0<\enorm{a-b}<2R,$ 
%     there is a curve $\Gamma$ in $A \cap \strconvsegment{a}{b}$ connecting $a$ and $b,$ and
%     whose length is at most the $R$-arclength of segment $ab,$ i.e. $\length{\Gamma}\leq 2R\arcsin{\frac{\enorm{a-b}}{2R}}.$
% \end{enumerate}

% \end{thm}
% \begin{proof} $(2)\Rightarrow (1)$. 
% For any   $a,b\in A$ with  $0<\enorm{a-b}<2R,$ there exists a curve $\Gamma$  in $A$ connecting $a$ and $b$ such that $\length{\Gamma}\leq 2R\arcsin{\frac{\enorm{a-b}}{2R}}.$  \Href{Lemmas}{lem:curve_separation} and \ref{lem:arc} imply that $\Gamma\bigcap \strconvsegment{a}{b}\neq \emptyset$. Thus, $A$ is proximally smooth with constant $R$ (see Ivanov's Theorem). 

%  $(1)\Rightarrow (2)$ is yielded by \Href{Lemma}{lem:shorter_curve_inside_str_convex_segment}. 
%  \end{proof}

This lemma provides a huge hint about the properties of 
 the shortest path connecting two sufficiently close points of a proximally smooth  set.
 Note that the shortest path exists in a finite-dimensional case.
 \begin{cor}\label{cor:shortest_curve_in_prox_smooth}
  Let $A\subset \hilbert$ be a closed, $R$-proximally smooth set, and let $\Gamma$ be a shortest curve
  in $A$ connecting $a$ and $b$ with $ \enorm{a-b} < 2R.$ Then $\length{\Gamma}$ is at most the $R$-arclength of $[a,b],$ i.e. $2R \arcsin\frac{\enorm{a-b}}{2R},$
  $\Gamma \subset \strconvsegment{a}{b},$ and $\Gamma$ is an $R$-proximally smooth set.
 \end{cor}
 \begin{proof}
 The inequality follows from \Href{Lemma}{lem:shorter_curve_inside_str_convex_segment}; the inclusion directly follows from \Href{Lemma}{lem:shorter_curve_inside_str_convex_segment}. Consequently, 
 the distance between any points $a_1$ and $b_1$ of $\Gamma$ is strictly less than $2R.$
 Thus, since $\Gamma$ is the shortest curve connecting $a$ and $b,$ 
 the part of $\Gamma$ connecting any two points $a_1$ and $b_1$ is the shortest curve in $A$ connecting them. Consequently, it belongs to $\strconvsegment{a_1}{b_1}.$ That is, by \Href{Proposition}{prp:equivalence2}, $\Gamma$ is an $R$-proximally smooth set. 
 \end{proof}

Since we will obtain a more general result than \Href{Theorem}{thm:existence of the shortest_curve}, we will need the following quantitative version of  \Href{Lemma}{lem:shorter_curve_inside_str_convex_segment}.
 \begin{lem}\label{lem:quantitatively_shorter_curve_inside_str_convex_segment}
  Let $A\subset \hilbert$ be a closed, $R$-proximally smooth set containing two distinct points $a$ and $b$ with $\enorm{a-b} < 2R.$ Fix $\epsilon \in \parenth{0, {\pi R - \rho_A{(a,b)}}}.$
    Consider a curve $\gamma$ in $A$ whose length is at most $\rho_A(a,b) + \epsilon.$ Then  the distance from an arbitrary point $p$ of $\gamma$ to $\strconvsegment{a}{b}$ is at most 
    $
    \max\{ \sqrt{4 \pi R \epsilon }, 4 \epsilon \}.
    $
\end{lem}
\begin{proof} 
By \Href{Lemma}{lem:simple_short_curve_in_prox_smooth_set}, 
$\pi R - \rho_A{(a,b)}$ is strictly positive.
Denote $\delta =   \max\{ \sqrt{4 \pi R \epsilon }, 4 \epsilon \}.$
Assume $\gamma$ contains a point $p$ at distance $\delta$ from 
$\strconvsegment{a}{b}.$ Then there are two moments $t_1$ and  $t_2$ such that
$0 \leq t_1 < t_2 \leq 1;$ $\gamma (t_1), \gamma (t_2) \in \strconvsegment{a}{b};$
the part of the curve between $\gamma (t_1)$ and 
$\gamma (t_2)$ contains $p$ and does not intersect the interior of
$\strconvsegment{a}{b}.$  Denote the part of $\gamma$ between these points by $\Gamma$.    Note that $\length\Gamma\le \length\gamma\le \varrho_A(a,b)+\epsilon < \pi R$, where the last inequality follows from
the choice of $\epsilon.$ 
By \Href{Lemma}{lem:curve_separation_stongly_convex_segment}, the length of $\Gamma$ is at least the $R$-arclength of $[\gamma(t_1),\gamma(t_2)]$ plus $2\epsilon.$ However, by \Href{Lemma}{lem:simple_short_curve_in_prox_smooth_set}, there is a curve $\tilde{\Gamma}$ in $\strconvsegment{a}{b} \cap A$ not longer than 
 the $R$-arclength of $[\gamma(t_1),\gamma(t_2)].$ 
 Substituting $\Gamma$ by $\tilde{\Gamma},$ one gets a curve in $A$
 of length at most $\rho_A(a,b) - \epsilon,$ which contradicts the definition of $\rho_A(a,b).$
\end{proof}

\section{Integration along circles and the second variation }
\label{sec:integration_along_circles}
To prove \Href{Theorem}{thm:existence of the shortest_curve}, we will need to bound the variation of the second order of the pointwise difference of two given curves with the same endpoints.
That is,  we will estimate a certain integral along a cycle. 
A discrete version of the desired inequality is the main result of this section.  In essence, 
it is a discrete version of the Wirtinger inequality, which says
that for a differentiable  function  of  $f \colon [0,1] \to \R$ with  
$f(0) = f(1) = 0$ whose derivative is $L_2$ integrable, the inequality
\begin{equation}\label{eq:wirtinger_inequality}
   \frac{ 1}{\pi^2}\int\limits_{[0,1]} {f^\prime}^2(t) \di t \geq 
   \int\limits_{[0,1]} {f}^2(t) \di t
\end{equation}
holds \cite[Inequality 257]{hardy1952inequalities}.

We also note that to obtain the tight bound in \Href{Theorem}{thm:existence of the shortest_curve}, we do need the best possible constant in the Wirtinger inequality, that is $\frac{1}{\pi^2}.$

  \begin{lem}
 \label{lem:lower_bound_on_s2_of_a_ciclyc_curve}
 Let $T= \{t_i\}_0^n$  with $0 = t_0 <  \dots < t_n = 1$ be a partition of $[0,1].$  
Consider a function $g \colon [0,1] \to X$ to be  Lipschitz with a  constant $L_g,$ and 
such that $g(0) = g(1)= 0,$ where $\parenth{X, \norm{\cdot}}$ is a normed space. 
Then 
 \[
   \parenth{\frac{\partvar[2]{g, T}}{\pi}}^2 
   \geq 
\sum\limits_{k \in [n]} \norm{g(t_k)}^2(t_k - t_{k-1}) - n  \parenth{L_g \mesh{T}}^2.
 \]
 \end{lem}
 \begin{proof}
 Denote $z_k = \norm{g(t_k)}$ for every $k \in [n],$ and $z_0 = 0.$ 
 By the triangle inequality,
 \begin{equation}\label{eq:1_bound_on_s2_of_a_ciclyc_curve}
     \abs{z_k - z_{k-1}} \leq \norm{g(t_k) - g(t_{k-1})} 
     \quad \text{for all} \quad k \in [n].
 \end{equation}
Define $f \colon [0,1] \to \R$ as
\[
f(t) = z_{k-1}+  (z_k - z_{k-1})\frac{t - t_{k-1}}{t_k - t_{k-1}} 
\quad \text{for all} \quad t \in [t_{k-1}, t_k] \quad \text{and} \quad  k \in [n] .
\]
That is, the graph of $f$ is a polyline in $\R^2$ with vertices $(t_i, z_i),$ 
$i \in [n] \cup \{0\}.$  By \eqref{eq:1_bound_on_s2_of_a_ciclyc_curve},
\begin{equation}\label{eq:v2_comparison_lower_bound_on_s2_of_a_ciclyc_curve}
    \partvar[2]{g, T} \geq \partvar[2]{f, T}.
\end{equation}
By construction,
$f(0)=f(1) =0.$ 
Hence, the  Wirtinger inequality \eqref{eq:wirtinger_inequality} yields 
\[
   \frac{ 1}{\pi^2}\int\limits_{[0,1]} {f^\prime}^2(t) \di t\geq 
   \int\limits_{[0,1]} {f}^2(t) \di t.
\]

Since $f^\prime(t) = \frac{z_k - z_{k-1}}{t_k - t_{k-1}}$ for  $ t\in (t_{k-1}, t_k),$
the left-hand side  is equal to
$   \parenth{\frac{\partvar[2]{f, T}}{\pi}}^2 .$ Using here inequality \eqref{eq:v2_comparison_lower_bound_on_s2_of_a_ciclyc_curve}, we see that  it suffices to show that
\[
\int\limits_{[0,1]} {f}^2(t) \di t\geq 
\sum\limits_{k \in [n]} \norm{g(t_k)}^2(t_k - t_{k-1}) - n  \parenth{L_g \mesh{T}}^2 = 
\sum\limits_{k \in [n]} z_k^2 (t_k - t_{k-1}) - n  \parenth{L_g \mesh{T}}^2
\]
 to complete the proof. The latter follows from the inequalities 
 \[
 \int\limits_{[t_{k-1},t_{k}]}f^2(t)\di t \geq z_k^2 (t_k - t_{k-1}) - L_g^2 (t_k - t_{k-1})^2 
 \quad \text{for every}\quad k \in [n];
 \]
which, in turn, follow from the inequalities
\[
f^2(t) \geq z_k^2 - L_g^2 (t_k - t_{k-1}) 
\quad \text{for all}\quad  t \in [t_{k-1},t_{k}]
\quad \text{for every}\quad k \in [n].
\]
Let us show that the latter inequalities hold.
If $z_k \leq L_g(t_k - t_{k-1}),$ then $z_k \leq L_g,$ and consequently,
\[
z_k^2 \leq  L_g^2  (t_k - t_{k-1})  \Leftrightarrow 
z_k^2 -  L_g^2  (t_k - t_{k-1}) \leq 0  \Rightarrow z_k^2 -  L_g^2  (t_k - t_{k-1}) \leq f^2(t).
\]
Consider the case $z_k \geq L_g(t_k - t_{k-1}).$ 
By construction and by inequality \eqref{eq:1_bound_on_s2_of_a_ciclyc_curve}, 
$f$ is a linear function on the interval $[t_{k-1},t_{k}]$ with $f(t_k) = z_k,$ whose slope is at most
$L_g.$ 
Thus,
$f(t)  \geq  z_k - L_g(t_k - t_{k-1}) \geq 0$ for all $t \in [t_{k-1},t_{k}].$ 
Hence, 
\[
 f^2(t) \geq  \parenth{z_k - L_g(t_k - t_{k-1})}^2 > z_k^2 - 2L_g z_k (t_k - t_{k-1}) 
 \quad \text{for all}\quad  t \in [t_{k-1},t_{k}].
\]
Again, by inequality \eqref{eq:1_bound_on_s2_of_a_ciclyc_curve} and since $z_0 = z_n =0$, 
$z_k \leq L_g \min\{t_k, 1 - t_k\} \leq \frac{L_g}{2}.$ That is, 
\[
 f^2(t)  \geq     z_k^2 - L_g^2  (t_k - t_{k-1}) 
\]
in this case as well.
The proof of the lemma is complete.
 \end{proof}

\begin{rem}
Using a more general inequality, namely Friedrichs's inequality, instead of the Wirtinger inequality, one can easily derive an upper bound on
$\sum\limits_{k \in [n]} \norm{g(t_k)}^p(t_k - t_{k-1})$ via 
 $  \parenth{\frac{\partvar[p]{g, T}}{C_p}}^p $ for all $p > 1,$
 where $C_p$ depends only on $p.$ 
\end{rem}

\section{Shortest path via discrete energy functional}
\label{sec:discrete_energy_functional}
Let us introduce our discrete energy functional.
\begin{dfn}\label{dfn:energy_functional}
Consider  $T= \{t_i\}_0^n$  with $0 = t_0 <  \dots < t_n = 1$  a partition of $[0,1]$ 
 and a finite sequence $X=\{x_i\}_0^n$ of points of $\hilbert$. We define 
 \[
 \mathcal{L} \! \parenth{T, X} = 
\sum_{i \in [n]} \frac{\enorm{x_{i} - x_{i-1}}^2}{t_{i} - t_{i-1}}.
 \]
\end{dfn}
The geometric meaning of  $ \mathcal{L} \! \parenth{T, X} $  is provided by the following example. Consider a function $f \colon [0, 1] \to \hilbert$ defined 
as $f(t_i) = x_i$ for $i \in [n] \cup \{0\}$ and extended continuously on every segment
$[t_{i-1}, t_i],$ $i \in [n].$ Then 
\begin{equation}
\label{eq:meaning_of_discrete_energy}
\mathcal{L} \! \parenth{T, X} = \partvarsq{f, T}.
\end{equation}

On the other hand, as an immediate consequence of the Cauchy--Schwarz inequality, we get
\begin{lem}\label{lem:discrete_energu_of_polyline}
     Let $T= \{t_i\}_0^n$  with $0 = t_0 <  \dots < t_n = 1$ be a partition of $[0,1]$ 
 and  $X=\{x_i\}_0^n$ be a finite sequence. Then $\mathcal{L} \! \parenth{T, X}$ is at least the squared length of the polyline $x_0 \dots x_n.$ In particular, $\mathcal{L} \!\parenth{T, X} \geq \enorm{x_0 - x_n}^2.$
\end{lem}
% \begin{proof}
%     Immediately follows from the Cauchy--Schwarz inequality.
% \end{proof}
This simple lemma together with the results introduced in \Href{Subsection}{subsec:curves} show that if one starts with a reasonable curve and then minimizes $\mathcal{L} \! \parenth{T, X}$ over all finite sequence $X$ of points from the curve, one gets a relatively fine approximation of the variation of the second order of the curve. And, moreover, the minimizer ``approximates'' the standard parametrization of the curve. 
Let us formalize this vague idea. 

\begin{dfn}
 Let $T= \{t_i\}_0^n$  with $0 = t_0 <  \dots < t_n = 1$ be a partition of $[0,1].$ 
 Let $A\subset\hilbert$ be a set with points $a$ and $b.$
 We define the \emph{discrete energy functional} as 
 \[
 \mathcal{L} \! \parenth{A,a,b, T} = 
 \inf \mathcal{L} \! \parenth{T, X},
 \]
 where the infimum is taken over all finite sequences $X=\{x_i\}_0^n \subset A$ with $x_0 = a$ and $x_n = b.$
\end{dfn}
To illustrate the power of this definition, consider the standard parametrization
$\gamma \colon [0,1] \to \hilbert$ of a rectifiable curve.  Then 
\[
 \mathcal{L} \! \parenth{\gamma,a,b,  \{t_i\}_0^n} \leq \mathcal{L} \! \parenth{ \{t_i\}_0^n,  \{\gamma(t_i)\}_0^n} \leq \length^2{\gamma}.
\]

Now, we will show that the problem of approximating the shortest curve  in a proximally smooth set is somewhat equivalent to the problem of finding the approximation of a minimizer of our disrcrete energy functional. 
The following result is  to illustrate the idea rather than obtaining the right dependencies. We prefer to hide under the rug technical details for the time-being since we plan to study our discrete energy functional in a separate paper.
\begin{thm}\label{thm:length_via_energy_func}
    Let $A$ be an $R$-proximally smooth set  containing distinct points $a$ and $b$ 
with $\enorm{a-b} < 2R,$ and  let $T= \{t_i\}_0^n$  with $0 = t_0 <  \dots < t_n = 1$ be a partition of $[0,1].$
Then 
\[
 {\rho^2_A(a,b)} \geq
 \mathcal{L} \! \parenth{A, a, b, T} \geq 
 \frac{\rho^2_A(a,b)}{k_T^2}, 
 \quad \text{where} 
 \quad 
 k_T = \frac{\arcsin\left( \frac{\pi}{2} \sqrt{\mesh T}\right)}{\frac{\pi}{2} \sqrt{\mesh T}}.
 \]
\end{thm}
Before we proceed with the proof, we want to note that 
$\frac{1}{k_T^{2}} = 1 - c \mesh T + \littleo{\mesh T}$ for a universal constant $c.$ That is, it does follow from the theorem that the minimizer of the discrete energy functional is a valid substitution to the shortest curve. We claim without proving it that for a sufficiently small $\mesh T,$ the dependence of the distortion is of the form
$1 - c \mesh^2 T + \littleo{\mesh^2 T}.$ 
\begin{proof}[Proof of \Href{Theorem}{thm:length_via_energy_func}]
    The leftmost inequality is simple. Indeed, consider a curve $\gamma$ in $A$ of length $\rho_{A}(a,b) + \epsilon$ connecting $a$ and $b$ given by its standard parametrization. Then, 
$
    \mathcal{L} \! \parenth{T, \{\gamma(t_i)\}_0^n}$
is less than $ \length^2{\gamma},
$ and the leftmost inequality follows.

    Let us obtain the rightmost inequality of the theorem.
    Consider a   finite sequence $\{x_i\}_0^n \subset A$ with $x_0 = a$ 
and $x_n = b$ and such that 
$
 \mathcal{L} \! \parenth{T, \{x_i\}_0^n}  < \mathcal{L} \! \parenth{A, a, b, T}  + \parenth{  \pi^2 R^2 - \rho^2_A(a,b)}.
$
Since $\rho_A(a,b) \leq 2R \arcsin \frac{\enorm{a -b}}{2R} < \pi R$ by 
\Href{Lemma}{lem:simple_short_curve_in_prox_smooth_set} and already obtained upper bound on $ \mathcal{L} \! \parenth{A, a, b, T},$ 
we conclude $\enorm{x_i - x_{i-1}}^2 \leq  \mathcal{L} \! \parenth{T, \{x_i\}_0^n}  \abs{t_i - t_{i-1}}$ for every $i \in [n],$ that is, 
$\enorm{x_i - x_{i-1}} \leq \pi R \sqrt{\mesh{T}}$ for every $i \in [n].$
Using \Href{Lemma}{lem:simple_short_curve_in_prox_smooth_set} again, 
one sees that there is a curve, say $\gamma_i,$ in $A$, connecting $x_i$ and $x_{i-1}$, whose length is at most the $R$-arclength of $\enorm{x_i - x_{i-1}}.$
Since $t \mapsto \frac{\arcsin t}{t}$ is monotonically increasing on $[0, \pi/2],$ 
$\length{\gamma_i} \leq k_T \enorm{x_i - x_{i-1}}.$ 
Concatenating all such $\gamma_i$ for all $i \in [n],$ we obtain a curve, say $\Gamma,$ in $A$ connecting $a$ and $b$ such that
\[
\rho_A(a,b) \leq \length{\Gamma} \leq k_T \sum\limits_{i \in [n]} \enorm{x_i - x_{i-1}} \leq k_T \sqrt{ \mathcal{L} \! \parenth{T, \{x_i\}_0^n}},
\]
where the last inequality follows from \Href{Lemma}{lem:discrete_energu_of_polyline}. Passing to the infimum, we get the desired inequality.
\end{proof}
% \begin{lem}
% Let $\Gamma$ be a rectifiable curve connecting $a$ and $b.$ 
% Then
% \[
% \parenth{\length{\Gamma}}^2 \geq \mathcal{L} \! \parenth{\Gamma,a,b, T} \geq \enorm{a-b}^2.
% \]
% \end{lem}
% \begin{lem}
%  Let $T= \{t_i\}_0^n$  with $0 = t_0 <  \dots < t_n = 1$ be a partition of $[0,1].$ 
% Let $A$ be a proximally smooth set of constant $R$ containing distinct points $a$ and $b$ 
% with $\enorm{a-b} < 2R.$ Let $\omega$ be a main $R$-arc of $[a,b].$
% Then
% \[
%  \mathcal{L} \! \parenth{\omega,a,b, T} \geq 
%   \mathcal{L} \! \parenth{A,a,b, T}.
% \]
% \end{lem}
% \begin{lem}
% Let $A$ be a proximally smooth set of constant $R$ containing distinct points $a$ and $b$ 
% with $\enorm{a-b} < 2R.$ Then
% \[
%  \rho_A(a,b)^2 \geq \mathcal{L} \! \parenth{A,a,b, T} \geq \enorm{a-b}^2
% \]
% \end{lem}

\subsection{Variation of the discrete energy functional}

The following lemma is the key ingredient of the proof of \Href{Theorem}{thm:existence of the shortest_curve}. We would say that the following result is essentially a digested version of G.E. Ivanov's key lemma in his proof of  
this theorem. We will need the following purely technical observation to get the desired asymptotic.
\begin{clm}\label{clm:bound_on_roots}
Let $s_1, s_2 \in [0, \beta/2]$ for some $\beta < 1.$
Then
\[
\parenth{\frac{\sqrt{1-s_1} + \sqrt{1-s_2}}{2}}^{-2} \leq 
1 + C_1(\beta)(s_1 + s_2), 
\quad \text{where} \quad 
C_1(\beta) = \parenth{\frac{1}{2} + \frac{3 \beta}{8 \parenth{1 -\beta}^{5/2} }}.
\]
\end{clm}
\begin{proof}
Using the arithmetic mean-geometric mean inequality in the first inequality of the chain, one gets
\[
    \parenth{\frac{\sqrt{1-s_1}+ \sqrt{1-s_2}}{2}}^2 
\geq 
\sqrt{1-s_1} \cdot \sqrt{1-s_2} = \sqrt{1 - \parenth{s_1 +s_2} + s_1 s_2 }
\geq 
\sqrt{1 - \parenth{s_1 +s_2}}.
\]
Denote $s= s_1 +s_2.$ So, the left-hand side of the desired inequality is at most $(1 - s)^{-1/2}$ and $s < 1.$
Using the Taylor series with Lagrange remainder, one gets
$
(1 - s)^{-1/2} = 1 + \frac{s}{2} +  \frac{3}{8} s^2 (1-s_0)^{-5/2} 
$
for some $s_0 \in [0, s].$ Particularly, $s_0 \leq \beta$ and $(1-s_0)^{-5/2} \leq 
(1-\beta)^{-5/2}.$ 
Thus, $(1 - s)^{-1/2} \leq 
1 + s \parenth{\frac{1}{2} + \frac{3 \beta}{8 \parenth{1 -\beta}^{5/2} }}.$
\end{proof}

The formulation of the following result is a bit cumbersome. However, we will justify the set of parameters just after the assertion of the result. Essentially, it shows that for two polylines inscribed in sufficiently short curves in a proximally smooth set, the metric projection of the average of their corresponding vertices is a shorter polyline inscribed in our set. 
\begin{lem}\label{lem:main_discrete_variation}
Fix $\lambda \in (0, 1/16].$
 Let $T= \{t_i\}_0^n$  with $0 = t_0 <  \dots < t_n = 1$ be a partition of $[0,1].$ 
Let $A$ be an $R$-proximally smooth  set containing distinct points $a$ and $b$ 
with $\enorm{a-b} < 2R.$ Suppose there are two sequences  $X= \{x_i\}_0^n$ and 
$Y=\{y_i\}_0^n $ of points of $A$ such that
\begin{enumerate}
    \item $x_0 = y_0 =a$ and  $x_n = y_n =b;$
    \item\label{ass1:main_thm} $\enorm{x_i - y_i} \leq 2r $ for some $r < R;$
    \item\label{ass2:main_thm} $\max\limits_{i \in [n]}\left\{\frac{\enorm{x_i - x_{i-1}}}{t_i - t_{i-1}}, \frac{\enorm{y_i - y_{i-1}}}{t_i - t_{i-1}}\right\}  \leq L.$
\end{enumerate} 
Denote $\Delta = \{\delta_i\}_0^{n},$ where 
$\delta_i= x_i - y_i,$   
%$\tilde{z}_i = \lambda x_i + (1-\lambda) y_i$
and 
$Z = \{z_i\}_0^n,$ where $z_i$ is the metric projection of $\lambda x_i + (1-\lambda) y_i$ onto $A,$  $i \in [n] \cup \{0\}.$ Then inequality 

\[
 \mathcal{L} \! \parenth{T, Z} \leq 
 \lambda  \mathcal{L} \! \parenth{T, X} + (1- \lambda)  \mathcal{L} \! \parenth{T, X} - 
 \lambda(1 -\lambda) \parenth{1 - \frac{L^2}{\pi^2 R^2} \parenth{1   + 20 \lambda}} \mathcal{L} \! \parenth{T, \Delta} + 
n  \parenth{\frac{L^2 \mesh{T}}{R}}^2
\]
holds.
\end{lem}
Before we proceed with the proof, let us shed some light  on the set of parameters as well. We start with the parameter $L.$ It can be set to be arbitrarily close to the length of the shortest curve by considering two curves $\gamma_1$ and $\gamma_2$ of length at most $\rho_A(a,b) + \epsilon$ and taking $x_i = \gamma_1(t_i)$ and $y_i = \gamma_2(t_i).$ By the already proven 
\Href{Lemma}{lem:short_curve_intro}, one can ensure that $L$ is strictly less than $\pi R -\epsilon.$   The second condition that
$\enorm{x_i - y_i} \leq 2r < 2R$ is required for our bound on the distance between a point of inscribed segment to its metric projection onto a proximally smooth set, that is, 
\Href{Lemma}{lem:projection_point_of_inscribed_segment}. \Href{Lemma}{lem:quantitatively_shorter_curve_inside_str_convex_segment} ensures us that for all suffieciently small $\epsilon,$ the corresponding inequality is met.   Thus, $\lambda$ can be chosen in such a way that the constant 
$ \lambda(1 -\lambda) \parenth{1 - \frac{L^2}{\pi^2 R^2} \parenth{1   + 20 \lambda}}$ is a positive constant independent of $\gamma_1$ and $\gamma_2.$ And we see that up to the last term, depending on the mesh of $T,$ our ``averaging'' of two polylines gives a shorter polyline.
\begin{proof}[Proof of \Href{Lemma}{lem:main_discrete_variation}]
%Denote $\Delta = \{\delta_i\}_0^n$, where $\delta_{i} = y_i - x_i.$
 Denoting $\tilde{z}_i = \lambda x_i + (1-\lambda) y_i, $ we get
    \[
    \enorm{\tilde{z}_i-\tilde{z}_{i-1}}^2 =
    \lambda \enorm{x_i - x_{i-1}}^2 + (1- \lambda) \enorm{y_i - y_{i-1}}^2  - 
    \lambda(1 -\lambda) \enorm{\delta_i - \delta_{i-1}}^2.
    \]
Denote $m_i =\frac{\enorm{z_i - \tilde{z}_{i}} + \enorm{z_{i-1} - \tilde{z}_{i-1}}}{2R}.$ 

By  \Href{Lemma}{lem:metric_projection_onto_weakly_convex_set}, 
$
\enorm{z_i  - z_{i-1}}^2 \leq  
\frac{1}{(1 - m_i)^2}  
\enorm{\tilde{z}_i-\tilde{z}_{i-1}}^2.
% \leq 
% \enorm{\tilde{z}_i-\tilde{z}_{i-1}}^2 \parenth{1 - \frac{\lambda(1-\lambda)}{R^2} 
% \max\{\enorm{\delta_i}^2, \enorm{\delta_{i-1}}^2\}}^{-1}.
$
Using \Href{Lemma}{lem:projection_point_of_inscribed_segment} and then \Href{Claim}{clm:bound_on_roots} (note that assumption (2) of our lemma implies the assumption $s_1, s_2\leq \beta/2$ in  \Href{Claim}{clm:bound_on_roots}), 
\[
\frac{1}{(1- m_i)^2} \leq 
\parenth{\frac{\sqrt{1-\frac{\lambda(1-\lambda)}{R^2}\enorm{\delta_i}^2}+ \sqrt{1-\frac{\lambda(1-\lambda)}{R^2}\enorm{\delta_{i-1}}^2}}{2}}^{-2}  
\leq 
1 +  \frac{\lambda(1-\lambda)}{R^2}\parenth{\enorm{\delta_i}^2 + \enorm{\delta_{i-1}}^2} \cdot C_1(\beta),
\]
where $C_1(\beta) = \parenth{\frac{1}{2} + \frac{3 \beta}{8 \parenth{1 -\beta}^{5/2} }}$ and $\beta = \frac{\lambda(1-\lambda)}{R^2}\parenth{\enorm{\delta_i}^2 + \enorm{\delta_{i-1}}^2}.$
By  assumption \eqref{ass1:main_thm} and the choice of $\lambda$, 
$\beta \leq 8 \lambda \leq 1/2.$ Thus,
$C_1 (\beta) \leq \frac{1}{2} + 20 \lambda.$ 
Denote
\[
C_2(\lambda) = 
 \frac{\lambda(1-\lambda)} {R^2}\parenth{ \frac{1}{2} + 20 \lambda}.
\]
Combining the obtained inequalities, we see that $\enorm{z_i  - z_{i-1}}^2$ is at most 
\[
\parenth{\lambda \enorm{x_i - x_{i-1}}^2 + (1- \lambda) \enorm{y_i - y_{i-1}}^2  - 
    \lambda(1 -\lambda) \enorm{\delta_i - \delta_{i-1}}^2}
    \parenth{1 +  C_2(\lambda)\parenth{\enorm{\delta_i}^2 + \enorm{\delta_{i-1}}^2} }.
\]
By assumption \eqref{ass2:main_thm}, $\lambda \enorm{x_i - x_{i-1}}^2 + (1- \lambda) \enorm{y_i - y_{i-1}}^2 
\leq L^2 (t_i - t_{i-1})^2.$ Hence, $\enorm{z_i  - z_{i-1}}^2$ is at most 
\[
\lambda \enorm{x_i - x_{i-1}}^2 + (1- \lambda) \enorm{y_i - y_{i-1}}^2  - 
    \lambda(1 -\lambda) \enorm{\delta_i - \delta_{i-1}}^2  +  
    L^{2} (t_i - t_{i-1})^2  C_2(\lambda) \parenth{\enorm{\delta_i}^2 + \enorm{\delta_{i-1}}^2}.
\]
% Thus,$ \frac{\enorm{z_i  - z_{i-1}}^2}{t_i - t_{i-1}} $ is at most
% \[
% \lambda \frac{\enorm{x_i - x_{i-1}}^2}{t_i - t_{i-1}} 
% + (1- \lambda) \frac{\enorm{y_i - y_{i-1}}^2}{t_i - t_{i-1}}  - 
%     \lambda(1 -\lambda) \enorm{\delta_i - \delta_{i-1}}^2  +  
%      C_\lambda(r) L^{2} (t_i - t_{i-1})  \enorm{\delta_i}^2 + C_3 (T,r, L)(t_i - t_{i-1}),
% \]
% where 
% \[
% C_3 (T,r, L) = L^2  C_2 (T,r, L) \parenth{\mesh{T}}^2.
% \]
Dividing by $t_i - t_{i-1}$ and summing up the inequalities, we get
\[
 \mathcal{L} \! \parenth{T, Z} \leq \lambda  \mathcal{L} \! \parenth{T, X} + (1- \lambda)  \mathcal{L} \! \parenth{T, X} - \lambda(1 -\lambda)  \mathcal{L} \! \parenth{T, \Delta} +  L^{2} C_2(\lambda)
 \sum\limits_{i \in [n]}\parenth{\enorm{\delta_i}^2 + \enorm{\delta_{i-1}}^2}(t_i - t_{i-1}).
\]

Now we will use \Href{Lemma}{lem:lower_bound_on_s2_of_a_ciclyc_curve} to deal with the sums 
$ \sum\limits_{i \in [n]}{\enorm{\delta_i}^2  }(t_i - t_{i-1})$ and $ \sum\limits_{i \in [n]}{ \enorm{\delta_{i-1}}^2}(t_i - t_{i-1}).$

Define the curve $\gamma_1 \colon [0,1] \to H$ on $T$ by 
$\gamma_1(t_i) = x_i, i\in [n] \cup \{0\}$ and extend it linearly to the whole segment $[0,1].$ Similarly, define $\gamma_2 \colon [0,1] \to H$ via the identities $\gamma_2(t_i) = y_i, i\in [n] \cup \{0\}.$
 Consider the function $g_1 \colon [0,1] \to \hilbert$ given by
\[
g_1(t) = \gamma_1 (t) - \gamma_2 (t).
\]
Clearly, it is Lipschitz with constant $2L,$ and
$g_1(0) = g_1(1) = 0.$ 
Thus, by equality (\ref{eq:meaning_of_discrete_energy}) and \Href{Lemma}{lem:lower_bound_on_s2_of_a_ciclyc_curve},
 \[\frac{ \mathcal{L} \! \parenth{T, \Delta}}{\pi^2} = \parenth{\frac{\partvar[2]{g_1, T}}{\pi}}^2 
   \geq 
\sum\limits_{k \in [n]} \enorm{g_1(t_k)}^2(t_k - t_{k-1}) - n  \parenth{2L \mesh{T}}^2 =
\]
\[
\sum\limits_{i \in [n]} \enorm{\delta_i}^2 (t_i - t_{i-1}) - n  \parenth{2L \mesh{T}}^2.
\]
Denote $T^\prime =\{t_i^\prime\}_0^n,$ where $t_i^\prime = 1 - t_{n-i}.$  
Consider the function $g_2 \colon [0,1] \to \hilbert$ given by
\[
g_2(t) = \gamma_1 (1-t) - \gamma_2 (1-t).
\]
Clearly, it is Lipschitz with constant $2L,$ and
$g_2(0) = g_2(1) = 0.$ Thus, by  \Href{Lemma}{lem:lower_bound_on_s2_of_a_ciclyc_curve},
\[\frac{ \mathcal{L} \! \parenth{T, \Delta}}{\pi^2} = \parenth{\frac{\partvar[2]{g_2, T^\prime}}{\pi}}^2 
   \geq 
\sum\limits_{k \in [n]} \enorm{g_2(t_k^\prime)}^2(t_{k}^\prime - t_{k-1}^\prime) - n  \parenth{2L \mesh{T^\prime}}^2 =
\]
\[
\sum\limits_{i \in [n]} \enorm{\delta_{i-1}}^2 (t_i - t_{i-1}) - n  \parenth{2L \mesh{T}}^2.
\]
Consequently, 
\[
 \mathcal{L} \! \parenth{T, Z} \leq 
 \lambda  \mathcal{L} \! \parenth{T, X} + (1- \lambda)  \mathcal{L} \! \parenth{T, X} - 
 \lambda(1 -\lambda) \parenth{1 - \frac{L^2}{\pi^2 R^2} \parenth{1   + 20 \lambda}} \mathcal{L} \! \parenth{T, \Delta} + 
n  \parenth{\frac{L^2 \mesh{T}}{R}}^2.
\]
\end{proof}

Let us use the just proven lemma to obtain the following result about the ``averaging of curves.''

\begin{cor}\label{cor:second_variation_bound_on_proj} Fix $\lambda \in (0, 1/16].$
    Let $A$ be an $R$-proximally smooth set containing distinct points $a$ and $b$ 
with $\enorm{a-b} < 2R.$ Let $\gamma_1, \gamma_2 \colon [0,1] \to A $ be two Lipschitz curves with constant  $L$ and such that 
$\enorm{\gamma_1(t) - \gamma_2(t)} \leq 2r $ 
for all $t \in[0,1] $
and some $r< R.$ Define $\Gamma_\lambda \colon [0,1] \to A $ as follows:
$\Gamma_\lambda(t)$ is the metric projection
of $\lambda \gamma_1(t) + (1 - \lambda)\gamma_2(t)$ onto $A.$ Then $\Gamma_\lambda$ is well-defined and
\[
\variationsq{\Gamma_\lambda} \leq 
 \lambda\variationsq{\gamma_1} + (1- \lambda) \variationsq{\gamma_2}- 
\lambda(1 -\lambda) \parenth{1 - \frac{L^2}{\pi^2 R^2} \parenth{1   + 20 \lambda}}
 \variationsq{\gamma_1 - \gamma_2}.
\]
\end{cor}
\begin{proof}
\Href{Lemma}{lem:projection_point_of_inscribed_segment} ensures that $\Gamma_\lambda$ is well-defined since by assumption $\enorm{\gamma_1(t) - \gamma_2(t)} \leq 2r < 2R$ 
for all $t \in[0,1]. $
Now take equipartition $T_n = \{\frac{i}{n}\}_{0}^{n}$  of $[0,1]$ and define $x_i = \gamma_1 (t_i)$ and $y_i = \gamma_2(t_i).$ Then the metric projection of $\lambda x_i + (1-\lambda) y_i$ onto $A$ is exactly $\Gamma_\lambda (t_i).$ By \Href{Lemma}{lem:main_discrete_variation} and by equation \eqref{eq:meaning_of_discrete_energy}, $\partvarsq{\Gamma_\lambda, T_n} $ is at most
\[
 \lambda\partvarsq{\gamma_1, T_n} + (1- \lambda) \partvarsq{\gamma_2, T_n}- 
\lambda(1 -\lambda) \parenth{1 - \frac{L^2}{\pi^2 R^2} \parenth{1   + 20 \lambda}}
\partvarsq{\gamma_1 - \gamma_2, T_n} + n  \parenth{L^2 \mesh{T_n}}^2.
\]
Using \Href{Lemma}{lem:variations_mesh_to_zero} and passing to the limit as $n \to \infty,$ we get the desired inequality.
\end{proof}

\subsection{Proof of \Href{Theorem}{thm:existence of the shortest_curve}}
We will prove a slightly stronger result than \Href{Theorem}{thm:existence of the shortest_curve}.
\begin{thm}\label{thm:existence of the shortest_curve2}
 Let $A$ be an $R$-proximally smooth  set  containing distinct points $a$ and $b$ 
with $\enorm{a-b} < 2R.$ The shortest curve in $A$ between $a$ and $b$ exists and is unique. It is an $R$-proximally smooth set   whose length is at most the $R$-arclength of $[a,b].$
Moreover, any minimizing sequence convergences to the shortest curve in the sense of second-order variation. 
\end{thm}
\begin{proof}
It follows from \Href{Corollary}{cor:shortest_curve_in_prox_smooth} and \Href{Lemma}{lem:second_variation_for_standard_parametrization} that it suffices to show that any minimizing sequence of curves is convergent to a curve in $A$ in the sense of second-order variation. 

 By \Href{Lemma}{lem:simple_short_curve_in_prox_smooth_set}, 
$ \rho_A (a,b) \leq 2R \arcsin{\frac{\enorm{a-b}}{2R}} < \pi R.$
Consider a minimizing sequence $\{{\gamma}_i\}_{i \in \N}$ of curves in $A.$ Consider $\epsilon$ such that
\[
\epsilon \in
\parenth{0, \min\left\{{\pi R} -  2R \arcsin \frac{\enorm{a-b}}{2R}, \frac{\pi R}{4} , \frac{1}{4\pi R}\left( R - \frac{\enorm{a-b}}{2}\right)^2\right\}}.
\]
Consider two curves $\gamma_i$ and $\gamma_j$ represented by their standard parametrization, and assume that they are not longer than 
$ \rho_A (a,b) + \epsilon.$ 

By the choice of $\epsilon,$ 
$\epsilon < {\pi R} -  2R \arcsin \frac{\enorm{a-b}}{2R}.$ 
Hence,  
 $ \rho_A (a,b) + \epsilon  \leq 2R \arcsin{\frac{\enorm{a-b}}{2R}} + \epsilon < \pi R.$
By \Href{Lemma}{lem:quantitatively_shorter_curve_inside_str_convex_segment}, 
the curves  $\gamma_i$ and $\gamma_j$ belong to the $\tilde{\epsilon}$-neighborhood of $\strconvsegment{a}{b},$ where $\tilde{\epsilon}=\max\left\{\sqrt{4\pi R\epsilon}, 4\epsilon\right\}$. 
Since $\epsilon < \frac{\pi R}{4},$ $\tilde{\epsilon} = \sqrt{4\pi R\epsilon}.$
Again, by the choice of $\epsilon,$ $\epsilon < \frac{1}{4\pi R}\left( R - \frac{\enorm{a-b}}{2}\right)^2;$ hence, $\tilde{\epsilon} < R - \frac{\enorm{a-b}}{2}.$
In particular, the distance between any two points of  $\gamma_i$ and 
$\gamma_j$ is not greater than  $\enorm{a-b} + 2\tilde{\epsilon} < 2R.$ Fix any $\lambda \in (0, 1/16]$. We may now use 
\Href{Corollary}{cor:second_variation_bound_on_proj}, which implies that
\[
\lambda(1 -\lambda) \parenth{1 - \frac{L^2}{\pi^2 R^2} \parenth{1   + 20 \lambda}}
 \variationsq{\gamma_i - \gamma_j}  \leq 
 \lambda\variationsq{\gamma_i} + (1- \lambda) \variationsq{\gamma_j}
 -\variationsq{\Gamma_\lambda},
\]
where $L$ can be chosen to be $\rho_A (a,b) + \epsilon$ and 
$\Gamma_\lambda$ is a curve in $A$ connecting $a$ and $b.$
Thus, the length of  $\Gamma_\lambda$ is at least $\rho_A (a,b).$ Consequently,
$ \variationsq{\Gamma_\lambda} \geq \rho_A^2 (a,b)$ by \Href{Lemma}{lem:second_variation_for_standard_parametrization}.
By the same lemma, $\lambda \variationsq{\gamma_i} + (1- \lambda) \variationsq{\gamma_j} \leq \parenth{\rho_A(a,b) + \epsilon}^2.$
So we conclude that
\[
\lambda(1 -\lambda) \parenth{1 - \frac{\parenth{\pi R + \rho_A(a,b)}^2}{4\pi^2 R^2} \parenth{1   + 20 \lambda}}
 \variationsq{\gamma_i - \gamma_j} \leq \parenth{\rho_A(a,b) + \epsilon}^2 - \rho_A^2(a,b).
\]
Thus, fixing any $\lambda \in (0, 1/16]$ such that the constant
\[
\frac{\parenth{\pi R + \rho_A(a,b)}^2}{4\pi^2 R^2} \parenth{1   + 20 \lambda}
\]
is strictly less than one 
(for example, $\lambda = \frac{1}{30} \min\{1, \frac{4\pi^2 R^2}{\parenth{\pi R + \rho_A(a,b)}^2} -1 \}$), we get that
\[
\variationsq{\gamma_i - \gamma_j} \leq C \epsilon
\]for a universal constant $C$ independent of  $\epsilon$ for all sufficiently small $\epsilon.$
That is, the sequence of curves  $\{\gamma_i\}_{i \in \N}$ is fundamental in the sense of  second order variation. By \Href{Lemma}{lem:second_variation_for_standard_parametrization}, it implies that this sequence is fundamental in $\operatorname{curve-dist}.$ Hence, by \Href{Lemma}{lem:convergences_in_variation}, it is convergent to a curve of the desired length. 
Since $A$ is closed the limit curve belongs to it. The theorem is proven.
\end{proof}
\subsection{Discussions}
We want to elaborate on the computation of $ \mathcal{L} \! \parenth{A,a,b, T}$ for a proximally smooth set $A.$ As a direct corollary of the construction provided by  \Href{Algorithm}{algo1}, one has
\begin{lem}
Let $A$ be an $R$-proximally smooth set containing two distinct points $a$ and $b$  with $\enorm{a-b} < 2R.$  Then for any partition $T$ of  $[0,1],$
the inequality $ \mathcal{L} \! \parenth{A,a,b, T} \leq \mathcal{L} \! \parenth{\omega,a,b, T}$ holds, where $\omega$ is  a main $R$-arc of $[a,b].$
\end{lem}
Note that $\mathcal{L} \! \parenth{\omega,a,b, T}$ is attained (by compactness) and is easy to compute. It follows, for example, from the following observation.

\begin{lem}\label{lem:three_point_energy_func}
Let $A$ be an $R$-proximally smooth set containing two distinct points $a$ and $b$  with $\enorm{a-b} < 2R.$ 
Fix $\lambda \in (0,1);$ denote $y_\lambda = (1-\lambda) a + \lambda b$ and $T_\lambda = \{0, \lambda, 1\}.$    The infimum  
$ \mathcal{L} \! \parenth{A,a,b, T_\lambda}$ is attained  on 
the unique sequence $\{a,x_\lambda, b\},$ where $x_\lambda$ is the metric projection of 
the point $y_\lambda$ onto $A.$
\end{lem}
\begin{proof}
    For any point $x,$ we have that
    \[
 \mathcal{L} \! \parenth{ T_\lambda, \{a,x, b\}}   = 
 \frac{(1-\lambda){\enorm{a-x}^2} + \lambda\enorm{x-b}^2}{\lambda(1-\lambda)}
  =
     \]
     \[
 \frac{\enorm{x}^2 - 2 \iprod{(1-\lambda)a + \lambda b}{x} + (1-\lambda)\enorm{a}^2 + \lambda \enorm{b}^2 }{\lambda(1-\lambda)} =     \frac{\enorm{x - y_\lambda}^2 - \enorm{y_\lambda}^2 + (1-\lambda)\enorm{a}^2 + \lambda \enorm{b}^2 }{\lambda(1-\lambda)}.
     \]
Note that since $\enorm{a-b} < 2R,$ the metric projection of 
the point $y_\lambda$ onto $A$ exists and is unique.
 The lemma follows.  
\end{proof}
If in \Href{Definition}{dfn:energy_functional} we consider that all of the points $x_i$, except for $x_1$, are fixed, then \Href{Lemma}{lem:three_point_energy_func} implies that $x_1$ is the metric projection of  $(1-\lambda)x_0 + \lambda x_2,$ where $\lambda=(t_1-t_0)/(t_2-t_0)$. So, \Href{Lemma}{lem:three_point_energy_func} completely defines the geometry of a minimizer of 
$\mathcal{L} \! \parenth{\omega,a,b, T}$ for any partition $T$ of $[0,1].$  For example, for the equipartition, we have the only  minimizer   of 
$\mathcal{L} \! \parenth{\omega,a,b, T}$ is the equipartition of the arc $\omega.$

Another useful trick in our pocket is \Href{Lemma}{lem:arclength}. It allows us to show that one can always choose a good approximation of the minimizer of  $ \mathcal{L} \! \parenth{A,a,b, T}$ from the strongly convex segment $\strconvsegment{a}{b}.$

\begin{lem}\label{lem:discrete_energy_radial_proj}
Let $A$ be a $R$-proximally smooth  set containing two distinct points $a$ and $b$  with $\enorm{a-b} < 2R;$ let $T=\{t_i\}_0^n$ be a partition of $[0,1].$ Then  for any $X=\{x_i\}_0^n \subset A$ with $x_0 = a$ and $x_n = b$ containing a point outside of 
$\strconvsegment{a}{b},$  there  is $Y=\{y_i\}_0^n \subset A \cap \strconvsegment{a}{b}$ with $y_0 = a$ and $y_n = b$ such that 
$ \mathcal{L} \! \parenth{Y, T} \leq \mathcal{L} \! \parenth{X, T}.$ 
\end{lem}
\begin{proof}
One can consequently substitute  points of $X$  from the outside of  
$\strconvsegment{a}{b}.$ Indeed, consider a  polyline $x_i \dots x_{i +m}$
such that $x_i$ and $x_{i+m}$ belong to $\strconvsegment{a}{b}$ but all intermediate points do not. Then by \Href{Lemma}{lem:arclength}, the $R$-arclength of this polyline is strictly longer than  the $R$-arclength of  $[x_i, x_{i+m}].$ Using \Href{Algorithm}{algo1}, we can substitute all the intermediate points by points inside $\strconvsegment{a}{b}$ non-increasing the length of each segment. The lemma follows.
\end{proof}

\subsection{Characterization of proximally smooth sets via short curves}
We want to note that the existence of a short curve inside a closed set connecting two  relatively close points is equivalent to the proximal smoothness of the set. 
\begin{cor}\label{cor:equivalence}
Let $A$ be a closed set in $\hilbert.$ 
The following conditions are equivalent:
\begin{enumerate}
\item The set  $A$ is $R$-proximally smooth. 
\item For all  $a$ and $b$ in $A$ with
$0 < \enorm{a-b} < 2R,$ the intersection of $A$ and $\strconvsegment{a}{b} \setminus\{a,b\}$ is non-empty. 
\item  For all  $a$ and $b$ in $A$ with
$0 < \enorm{a-b} < 2R,$ there  is a curve of length at most the $R$-arclength of segment $[a,b],$
that is $2R \arcsin \frac{\enorm{a-b}}{2R},$  connecting $a$ and $b$ in $A.$
\item  For all  $a$ and $b$ in $A$ with
$0 < \enorm{a-b} < 2R,$ there  exists and is unique the shortest curve in $A$  connecting $a$ and $b;$ and  the shortest curve is not longer than $2R \arcsin \frac{\enorm{a-b}}{2R}.$
\end{enumerate}
\end{cor}
\begin{proof}
The first two conditions are equivalent by  \Href{Proposition}{prp:equivalence2}.
\Href{Lemma}{lem:short_curve_intro} yields the implication $(1) \Rightarrow (3),$
 the implication $(3) \Rightarrow (2)$ directly follows from  \Href{Corollary}{cor:main_arc}. The implication $(4) \Rightarrow (3)$ is trivial, the implication 
 $(1) \Rightarrow (4)$ follows from \Href{Theorem}{thm:existence of the shortest_curve2}.
\end{proof}

\section{Characterization of proximal smooth curves}
\label{sec:prox_smooth_curve_characterization}
The goal of the current section is to prove \Href{Theorem}{thm:characterization_of_prox_reg_curves}.
\subsection{Properties of the Euclidean spheres and circles}

We will need the following technical trivial claims. 
\begin{clm}\label{claim:equivalnce_of_metrics_on_sphere}
A function  from $[a,b]$ to $\sphere_{\hilbert}$ is Lipschitz with a constant $L$ on $[a,b]$ 
if and only if it is Lipschitz with constant $L$ on $[a,b]$ in the intrinsic metric of the unit sphere 
$\sphere_{\hilbert}.$
\end{clm}

% As we will see later,  the curves in a Hilbert space whose natural parametrization is Lipschitz differentiable are exactly proximally smooth curves. We will need some properties of them later.  However, the results of the current sections do not use any knowledge of proximally smooth sets at all.  

\begin{clm}
Let $u$ be a unit vector. 
If a function  $f \colon [a,b] \to \sphere_{\hilbert}$ is Lipschitz with constant $L$ on $[a,b], $ then
the function $g \colon [a,b] \to \R$ given  by $g(t)=\anglef{u}{f(t)}$ is Lipschitz with constant $L$ on 
$[a, b].$ 
\end{clm}
\begin{proof}
By \Href{Claim}{claim:equivalnce_of_metrics_on_sphere}, 
\[
 \anglef{f(t_0)}{f(t_1)}  \leq L \abs{t_0 - t_1}.
\]
% Since $t - \frac{t^3}{6} \leq \sin t$ on $t \in \left[0, \frac{\pi}{2}\right]$ and 
% $\frac{\anglef{\gamma^\prime(t_0)}{\gamma^\prime(t_1)}}{2} \leq \frac{\pi}{2},$
% we get that
By the triangle inequality for angles, 
$
\abs{\anglef{u}{f(t_0)} - \anglef{u}{f(t_1)}} \leq \anglef{f(t_0)}{f(t_1)}.
$
The claim follows. 
\end{proof}
\begin{clm}\label{claim:cos_bound_lipschitz_map+to+sphere}
If a function  $f \colon [a,b] \to \sphere_{\hilbert}$ is Lipschitz with constant $L$ on $[a,b], $
then 
\[
\iprod{f(t_1)}{f(t_2)} \geq \cos L(t_2 - t_1)
\]
for all $t_1, t_2 \in [a,b] $ such that
$L \abs{t_2 - t_1} \leq  \pi.$
\end{clm}
\begin{proof}
    Take $u = f(t_1)$ in the previous claim. 
    We get that  $\anglef{f(t_1)}{f(t_2)} \leq  L\abs{t_2 - t_1} \leq \pi.$
    Hence, $\iprod{f(t_1)}{f(t_2)} = \cos \anglef{f(t_1)}{f(t_2)}  \geq \cos L(t_2 - t_1).$
\end{proof}
\begin{clm}\label{clm:integral_on_arc}
    Consider a function $f \colon [0, 2 \pi R] \to \sphere_{\R^2}$   defined by
    $f(t) = (\sin \frac{t}{R}, \cos \frac{t}{R}).$ Denote $a = (0,0),$ $c = (R,0),$  
    $b = \int\limits_{0}^{\tau} f(t) \di t$ for some $\tau \in [0, 2\pi].$
    Then $b$  is the point on the circle of radius $R$ centered at $c$ and such that the angle between $a - c$ and $b-c$  in the clock-wise direction is equal to $\frac{\tau}{R}.$
\end{clm}
\begin{proof}
By direct calculations,
    \[
    \int\limits_{0}^{\tau} f(t) \di t = 
    \parenth{
    \int\limits_{0}^{\tau} \sin \frac{t}{R} \di t, \int\limits_{0}^{\tau} \cos \frac{t}{R} \di t =
       }=
      R \! \parenth{
      1 -\cos \frac{\tau}{R}  ,    \sin \frac{\tau}{R}
       }.
    \]
    The claim follows.
\end{proof}
\subsection{Integral curves}
% \begin{lem}\label{lem:def_of_integral_curve}
% Let $f$ be a Lipschitz function from a non-trivial segment $[t_0,t_1]$ to the unit sphere 
%  $(\sphere_{\hilbert})$ of $\hilbert.$
% Then there exists a unique curve $g \colon [t_0,t_1] \to \hilbert$ such that $g(t_0) = 0$ and 
% $g^\prime(t) = f(t)$ for all $t \in [t_0, t_1].$
% \end{lem}

   Let $f:[t_0,t_1]\to\hilbert$ be a Lipschitz function. We call {\it the integral curve of a Lipschitz function} the curve $g:[t_0, t_1]\to\hilbert$ such that $g=\int\limits_{[0,\tau]}f(t)\di t,$ where we mean the Bochner integral. We refer the reader interested in the theory of the Bochner integral to \cite{hille1996functional} and Theorem 3.4.1 therein. 

\begin{lem} \label{lem:natural_parametrization_Lipschitz_bound}
Fix $T \in (0,  \pi R].$
Assume $f \colon [0, T] \to \sphere_{\hilbert}$ is  a Lipschitz function with a constant $\frac{1}{R}$ on $[0,T].$ 
Define $g(\tau) = \int\limits_{[0, \tau]} f(t) \di t .$ Then 
\[
\enorm{g(0) - g(\tau)} \geq  2R \ \sin \frac{\tau}{2R}
\]
for all $\tau \in [0, T].$ Moreover, if the identity holds for some $\tau,$ then the part of 
$g$ between $g(0)$ and $g(\tau)$ is the arc of a circle of radius $R$ passing through these points. 
\end{lem}
\begin{proof}
By assumptions, $\frac{\abs{t_1 - t_2}}{R} \leq \pi$ for any $t_1, t_2 \in [0,T].$
Thus, by \Href{Claim}{claim:cos_bound_lipschitz_map+to+sphere}, 
\[
\enorm{g(0) - g(\tau)}^2 = \iprod{\int\limits_{0}^{\tau} f(t_1) \di t_1 }{\int\limits_{0}^{\tau} f(t_2) \di t_2}
= \int\limits_{0}^{\tau}  \di t_1 \int\limits_{0}^{\tau}  \iprod{f(t_1)}{f(t_2)} \di t_2  \geq
\]
\[
\int\limits_{0}^{\tau}  \di t_1 \int\limits_{0}^{\tau} \cos \frac{t_2 - t_1}{R} \di t_2 = 
R \int\limits_{0}^{\tau} \parenth{ \sin \frac{\tau - t_1}{R} + \sin \frac{ t_1}{R}}\di t_1 =
2 R^2 \parenth{1 - \cos \frac{\tau}{R}} =4 R^2 \sin^2 \frac{\tau}{2R}.
\]
The inequality follows. 
The identity is attained if and only if $\anglef{f(t_1)}{f(t_2)} = \frac{\abs{t_2 - t_1}}{R}$ for all
$t_2, t_1 \in [0, \tau].$ By the triangle inequality for angles, we see that the part of $f$
between $0$ and $\tau$ is an arc of unit circle on which $f$ goes with the constant speed. The identity case follows from \Href{Claim}{clm:integral_on_arc}.
 \end{proof}
 \subsection{Proof of \Href{Theorem}{thm:characterization_of_prox_reg_curves}}

The proof of the theorem consists of several steps. 
We start by showing the implication \ref{i1:thm_curves_char} $\Rightarrow$  \ref{i2:thm_curves_char}. 

\begin{clm}\label{clm:curve_proximal_smoothness}
Let $a$ and $b$ be points of $\hilbert$ such that  $0 < \enorm{a-b} \leq 2R.$
Let $\Gamma$ be a curve without self-intersection connecting points $a$ and $b.$  If $\Gamma$ is an $R$-proximally smooth set, then for any two points $a_1$ and $b_1$ on $\Gamma,$ the part of 
$\Gamma$ between $a_1$ and $b_1$ belongs to 
$\strconvsegment{a_1}{b_1}$ and the length of this part is at most the $R$-arclength between $a_1$ and $b_1.$
\end{clm}
\begin{proof}
By  \Href{Lemma}{lem:simple_short_curve_in_prox_smooth_set}, 
$\Gamma$ as a set contains a curve $\Gamma_1$ connecting $a_1$ and $b_1$ such that $\Gamma_1 \subset \strconvsegment{a_1}{b_1}$ and whose length is at most the $R$-arclength between 
$a_1$ and $b_1.$
Since $\Gamma$ has no self-intersection, we conclude that the part of 
$\Gamma$ between $a_1$ and $b_1$  coincides with $\Gamma_1.$
\end{proof}

\begin{lem}\label{lem:curve_length_to_Lipschitz}
Let $a$ and $b$ be points of $\hilbert$ such that  $0 < \enorm{a-b} \leq 2R.$
Let $\Gamma$ be a curve without self-intersection connecting points $a$ and $b.$  
 If $\Gamma$ is an $R$-proximally smooth set, then 
 the natural parametrization of $\Gamma$ is differentiable, and its derivative is  Lipschitz with constant $\frac{1}{R}.$
\end{lem}
\begin{proof}
Clearly, the assumptions together with the previous claim imply that $\Gamma$ is of finite length. Let
$x \colon [0, \length{\Gamma}]\to \Gamma$ be the natural parametrization of $\Gamma.$
By \Href{Claim}{clm:curve_proximal_smoothness}, for any $0 \leq t_1 < t_2 \leq \length{\Gamma},$ one has
\[
\enorm{x({t_2}) - x({t_1})} \leq t_2 - t_1 \leq 2R \arcsin \frac{\enorm{x({t_2}) - x({t_1})}}{2R}.
\]
Hence,
\[
1 \geq \frac{\enorm{x({t_2}) - x({t_1})}}{t_2 - t_1} \geq \frac{2R}{{t_2 - t_1}}
\sin \frac{{t_2 - t_1}}{2R}.
\]
% \[\text{if we will need}
% \geq  1 - \frac{1}{6}\parenth{\frac{\enorm{x({t_1}) - x({t_2})}}{2R}}^2.
% \]
Consequently, 
\begin{equation}\label{eq:norm_convergence_lem:curve_length_to_Lipschitz}
\frac{\enorm{x({t_1}) - x({t_2})}}{t_1 - t_2} \rightrightarrows 1  
\quad \text{as} \quad t_2 \to t_1 \quad \text{uniformly on} \quad  [0, \length{\Gamma}].
\end{equation}

 Now fix $t_1, t_2, \tau$ such that $0 \leq t_1 < \tau <t_2 \leq \length{\Gamma}.$ 
 By  the assumptions of the lemma and using \Href{Lemma}{lem:angle_inside_strongly_convex_segment},
 we have that 
 \begin{equation}\label{eq:angle_convergence_lem:curve_length_to_Lipschitz}
 \anglef{x(\tau) - x(t_1)}{x(t_2) - x(t_1)} \leq  \arcsin \frac{\enorm{x(t_2) - x(t_1)}}{2R}\leq  \arcsin \frac{t_2- t_1}{2R}.
 \end{equation}
 
 Fix any $t_0\in [0,\length\Gamma)$ and a sequence $t_k\to t_0$, $t_k>t_0$. It follows by \eqref{eq:angle_convergence_lem:curve_length_to_Lipschitz} that for any $i,k\in\N$
\[
 \anglef{x(t_i) - x(t_0)}{x(t_k) - x(t_0)}\leq \arcsin \frac{\max\{t_i- t_0,t_k-t_0\}}{2R}.
\] 
 So, $\left\{ \frac{x(t_k) - x(t_0)}{|x(t_k) - x(t_0)|}\right\}$ is a Cauchy sequence and, consequently, converges.

This assertion  together with \eqref{eq:norm_convergence_lem:curve_length_to_Lipschitz} imply that the right hand derivative of  $x$ is well-defined on $[0, \length{\Gamma})$. Similarly, the left hand derivative of  $x$ is well-defined on $(0, \length{\Gamma}]$.

It follows from \eqref{eq:angle_convergence_lem:curve_length_to_Lipschitz} that
\[
 \anglef{x^\prime(t_1)}{x(t_2) - x(t_1)}\leq \arcsin \frac{t_2-t_1}{2R}.
\] 
Using the triangle inequality for angles, we get
\[
\anglef{x^\prime(t_1)}{x^\prime(t_2)} \leq 2 \arcsin \frac{\abs{t_2 - t_1}}{2R},
\]
and since $\enorm{x^\prime(t_2)}=\enorm{x^\prime(t_2)} =1,$
we conclude that $\enorm{x^\prime(t_2) - x^\prime(t_1)} \leq  \frac{\abs{t_2 - t_1}}{R}.$
The lemma is proven.
\end{proof}

Now we show the implication \ref{i2:thm_curves_char} $\Rightarrow$  \ref{i1:thm_curves_char} of \Href{Theorem}{thm:characterization_of_prox_reg_curves}
\begin{lem}
   Let $\Gamma\colon [0,\length{\Gamma}] \to \hilbert$ be a rectifiable curve given by its natural parametrization such that $\length{\Gamma} \leq \pi R.$ 
   If the natural parametrization of $\Gamma$ is differentiable, and its derivative is  Lipschitz  with constant  $\frac{1}{R},$ 
then 
$\Gamma$ is an $R$-proximally smooth set.
\end{lem}
\begin{proof}
% If $\operatorname{length} = \pi R, $ then $\enorm{a-b} = 2R$ and $Gamma$ is a main 
% $R$-arc of segment $ab$ by \Href{Lemma}{lem:natural_parametrization_Lipschitz_bound} and the equality case therein. Thus, the lemma follows in this case.

% Consider the case
Take any two points $\Gamma(t_1), \Gamma(t_2)$ such that $\enorm{\Gamma(t_1)- \Gamma(t_2)} \leq 2R.$  
By \Href{Lemma}{lem:natural_parametrization_Lipschitz_bound}, the length of the part of $\Gamma$ between these points
is at most the $R$-arclength between these points, i.e.
$
2R \arcsin \frac{\enorm{\Gamma(t_1)-\Gamma(t_2)}}{2R}.
$
By \Href{Corollary}{cor:main_arc}, this part of $\Gamma$ intersects $\strconvsegment{\Gamma(t_1)}{\Gamma(t_2)}.$ Thus, $\Gamma$ is proximally smooth by \Href{Proposition}{prp:equivalence2}.
\end{proof}
% \begin{lem}
% Let $A \subset \hilbert$ be a closed, $R$-proximally smooth set that can be covered by a ball of radius 
%  $r \leq R.$ Then $A$ is contractible.
% \end{lem}
% \begin{proof}
% Trivial --- contained in one of my uvo-modul papers.
% \end{proof}

\bibliographystyle{alpha}
\bibliography{uvolit_text} 

\begin{thebibliography}{GKPS99}

\bibitem[BBI22]{burago2022course}
Dmitri Burago, Yuri Burago, and Sergei Ivanov.
\newblock {\em A course in metric geometry}, volume~33.
\newblock American Mathematical Society, 2022.

\bibitem[BG87]{borwein1987proximal}
Jonathan~M. Borwein and J.R. Giles.
\newblock The proximal normal formula in banach space.
\newblock {\em Transactions of the American Mathematical Society},
  302(1):371--381, 1987.

\bibitem[BI09]{IvBal2009}
M.~V. Balashov and G.~E. Ivanov.
\newblock Weakly convex and proximally smooth sets in {B}anach spaces.
\newblock {\em Izv. RAN. Ser. Mat.}, 73(3):23--66, 2009.

\bibitem[BLNP07]{bezdek2007ball}
K{\'a}roly Bezdek, Zsolt L{\'a}ngi, M{\'a}rton Nasz{\'o}di, and Peter Papez.
\newblock Ball-polyhedra.
\newblock {\em Discrete \& Computational Geometry}, 38:201--230, 2007.

\bibitem[CG67]{coxeter1967geometry}
Harold Scott~Macdonald Coxeter and Samuel~L. Greitzer.
\newblock {\em Geometry revisited}, volume~19.
\newblock Maa, 1967.

\bibitem[CSW95]{Clarke1995}
F.~H. Clarke, R.~J. Stern, and P.~R. Wolenski.
\newblock Proximal {S}moothness and {L}ower--${C}^{2}$ {P}roperty.
\newblock {\em J. Convex Anal.}, 2(1):117--144, 1995.

\bibitem[CT10]{Colombo_thibault}
G.~Colombo and L.~Thibault.
\newblock Prox-regular sets and applications.
\newblock In {\em Handbook of Nonconvex Analysis}, pages 99--182. Int. Press,
  Somerville, MA., 2010.

\bibitem[GKPS99]{gromov1999metric}
Mikhael Gromov, Misha Katz, Pierre Pansu, and Stephen Semmes.
\newblock {\em Metric structures for Riemannian and non-Riemannian spaces},
  volume 152.
\newblock Springer, 1999.

\bibitem[HLP52]{hardy1952inequalities}
Godfrey~Harold Hardy, John~Edensor Littlewood, and George P{\'o}lya.
\newblock {\em Inequalities}.
\newblock Cambridge university press, 1952.

\bibitem[HP96]{hille1996functional}
Einar Hille and Ralph~Saul Phillips.
\newblock {\em Functional analysis and semi-groups}, volume~31.
\newblock American Mathematical Soc., 1996.

\bibitem[IL22]{ivanov2022rectifiable}
Grigory Ivanov and Mariana~S. Lopushanski.
\newblock Rectifiable curves in proximally smooth sets.
\newblock {\em Set-Valued and Variational Analysis}, 30(2):657--675, 2022.

\bibitem[Iva06]{IvMonEng}
G.~E. Ivanov.
\newblock {\em Weakly Convex Sets and Functions. Theory and Applications. (in
  Russian)}.
\newblock Moscow, 2006.

\bibitem[Iva15]{Ivanov_2015}
G.~E. Ivanov.
\newblock Sharp estimates for the moduli of continuity of metric projections
  onto weakly convex sets.
\newblock {\em Izvestiya: Mathematics}, 79(4):668, aug 2015.

\bibitem[Iva17]{Ivanov:233858}
G.~M. Ivanov.
\newblock Hypomonotonicity of the normal cone and proximal smoothness.
\newblock {\em Journal of Convex Analysis}, 24(4):27. 1313--1339, 2017.

\bibitem[PRT00]{Pol_Rock_Thi1}
R.~Poliquin, R.~Rockafellar, and L.~Thibault.
\newblock Local differentiability of distance functions.
\newblock {\em Trans. Amer. Math. Soc.}, 352(11):5231--5249, 2000.

\bibitem[Roc81]{rockafellar1981favorable}
R.~T. Rockafellar.
\newblock Favorable classes of {L}ipschitz continuous functions in subgradient
  optimization.
\newblock 1981.

\bibitem[Sch76]{schaffer1976geometry}
Juan~Jorge Sch\"{a}ffer.
\newblock {\em Geometry of spheres in normed spaces}, volume No. 20 of {\em
  Lecture Notes in Pure and Applied Mathematics}.
\newblock Marcel Dekker, Inc., New York-Basel, 1976.

\end{thebibliography}

\end{document}